\documentclass[a4paper,11pt]{amsart} 

\usepackage{amsmath} 
\usepackage{amssymb}
\usepackage{amsthm} 
\usepackage{graphicx} 
\usepackage{enumerate}
\usepackage{amsfonts} 
\usepackage{ulem} 
\usepackage{float} 
\usepackage{hyperref} 
\usepackage{cite} 
\usepackage[dvipsnames]{xcolor} 
\usepackage{tikz} 
\usepackage{tikz-cd} 
\usepackage{mathabx} 
\usepackage{faktor} 
\usepackage{array} 
\usepackage{transparent} 
\usepackage{mathdots} 
\usepackage{adjustbox} 
\usepackage{thmtools} 
\usepackage{thm-restate} 
\usepackage[margin=1.25in]{geometry} 
\usepackage{mathtools} 

\usetikzlibrary{decorations.markings} 

\theoremstyle{plain} 
\newtheorem{thm}{Theorem}[section] 
\newtheorem{lemma}[thm]{Lemma} 
\newtheorem{prop}[thm]{Proposition} 
\newtheorem{cor}[thm]{Corollary} 

\theoremstyle{definition} 
\newtheorem{defn}[thm]{Definition} 
\newtheorem{notation}[thm]{Notation}

\theoremstyle{remark} 
\newtheorem*{rem}{Remark} 

\newtheorem{obs}[thm]{Observation} 

\DeclareMathOperator{\aut}{Aut} 
\DeclareMathOperator{\out}{Out} 
\DeclareMathOperator{\outs}{\out_{\mathfrak{S}}} 
\DeclareMathOperator{\auts}{\aut_{\mathfrak{S}}} 
\DeclareMathOperator{\inn}{Inn} 
\DeclareMathOperator{\stab}{Stab} 

\DeclareMathOperator{\dash }{\textbf{\textemdash}} 
\DeclareMathSymbol{\shortminus}{\mathbin}{AMSa}{"39} 
\newcommand{\coset}{{\cdot}} 

\newcommand{\factorautos}{\displaystyle\prod_{i=1}^{n}\aut(G_{i})} 

\newcommand{\spoke}[2]{\mathrm{Sp}_{#1}(#2)} 
\newcommand{\vol}[2]{|#2|_{#1}} 
\newcommand{\tree}[1]{\widetilde{#1}} 
\newcommand{\class}[1]{[#1]} 
\newcommand{\funalpha}{\underline{\alpha}} 
\newcommand{\funA}[1]{\underline{A^{#1}}} 

\tikzset{->-/.style={decoration={markings, mark=at position .55 with {\arrow{Classical TikZ Rightarrow}}}, postaction={decorate}}} 
\tikzset{-<-/.style={decoration={markings, mark=at position .55 with {\arrow{Classical TikZ Leftarrow}}}, postaction={decorate}}} 

\allowdisplaybreaks

\title{Generators for the Pure Symmetric Outer Automorphisms of a Free Product}
\author{Harry M J Iveson}

\begin{document}

\maketitle


\begin{abstract}
We construct a `nice' subcomplex of the Outer Space for a free product in order to give a geometric proof that the pure symmetric outer automorphisms of a given splitting of a free product are generated by factor outer automorphisms and Whitehead outer automorphisms relative to the splittting.
\end{abstract}

\section*{Introduction}


Fouxe-Rabinovitch \cite{F-R1940} in 1940 gave a presentation for the automorphism group of a free product of finitely many freely indecomposable factors.
In 1996, McCullough--Miller \cite{McCullough1996} considered the symmetric automorphisms of a free product relative to a splitting where the factors need not be indecomposable.
Note that if a splitting is freely indecomposable, then every automorphism is a symmetric automorphism, so McCullough--Miller \cite{McCullough1996} is in some sense a generalisation of Fouxe-Rabinovitch \cite{F-R1940}.
In this paper, we provide an alternative method for determining generators of the pure symmetric outer automorphism group (described in Subsection \ref{subsection pure symmetric autos}) of a given splitting of a free product.
We show that this group is generated by the subgroup comprising factor outer automorphisms (described in Subsection \ref{subsection factor autos}) relative to the splitting and Whitehead outer automorphisms (described in Subsection \ref{subsection whitehead autos}) relative to the splitting.

\begin{restatable*}[Main Theorem]{thm}{maintheorem}\label{main theorem}
Let $G=G_{1}\ast\dots\ast G_{n}$ be a group which splits as a free product of $n\ge3$ non-trivial factors, and let $\mathfrak{S}=(G_{1},\dots,G_{n})$.
Then any pure symmetric outer automorphism $\Psi\in\outs(G)$ of the splitting $G_{1}\ast\dots\ast G_{n}$ can be written as a product of factor outer automorphisms relative to $\mathfrak{S}$ and Whitehead outer automorphisms relative to $\mathfrak{S}$.
\end{restatable*}

We do this by showing that a particular subcomplex $\mathcal{S}_{n}$ (see Definition \ref{defn Sn}) of Guirardel--Levitt's \cite{Guirardel2007} Outer Space is path connected.

\begin{restatable*}{cor}{connected}\label{cor Sn path connected}
The subcomplex $\mathcal{S}_{n}$ is path-connected.
\end{restatable*}

Our key ingredients are Proposition \ref{prop alpha-A paths} and Proposition \ref{prop inductive step for generators}.
The former is a statement about our subcomplex $\mathcal{S}_{n}$ of Outer Space, detailing how one may find certain paths within the subcomplex.
The latter turns this into a statement about (outer) automorphisms, and how Whitehead and factor (outer) automorphisms are sufficient to traverse these paths, and hence the subcomplex.

The paper begins with some background defintions and preliminary results.
The varying types of automorphism we will encounter are introduced in Section \ref{section automorphisms}, although for the most part these will not be used until Section \ref{section generators}, where we prove our main theorem.
In order to discuss the `Outer Space'  for our free product splitting (Subsection \ref{subsection outer space}), we present some Bass--Serre theory regarding graphs of groups in Section \ref{section graphs of groups}.
In Section \ref{section Sn} we introduce our subcomplex $\mathcal{S}_{n}$ of Outer Space relative to our chosen splitting of our free product, and discuss some of its basic properties.
The key ideas in this paper are found in Section \ref{section connectedness}, where we use a reduction argument to show that our subcomplex $\mathcal{S}_{n}$ is path connnected.
Finally in Section \ref{section generators} we use this result to determine generators for the pure symmetric outer automorphism group of our free product splitting.

\setcounter{tocdepth}{1}
\tableofcontents


\section{Automorphisms of $G_{1}\ast\dots\ast G_{n}$}\label{section automorphisms}

Below is a summary of definitions and notation regarding automorphisms
adapted from the notation of Gilbert \cite{Gilbert1987} and McCool \cite{McCool1986}.

Throughout, we consider a group $G$ which splits as a free product $G_{1}\ast\dots\ast G_{n}$, where each $G_{i}$ is non-trivial and $n\ge3$.
We refer to each $G_{i}$ as a \emph{factor group}.

\begin{notation}
Let $G$ be a group.
We denote by $\out(G)$ the outer automorphism group of $G$.
Note that $\out(G)=\faktor{\aut(G)}{\inn(G)}$, where
$\aut(G)$ is the group of automorphisms of $G$ and 
$\inn(G)$ is the normal subgroup of $\aut(G)$ comprising all inner automorphisms of $G$,
that is, automorphisms of the form $g\mapsto h^{-1}gh=:g^{h}$ for some fixed $h\in G$.
\end{notation}

We will usually write elements of $\out(G)$ as upper-case Greek letters and their $\aut(G)$ representatives as their lower-case counterparts, although if we are given $\psi\in\aut(G)$ we may sometimes write $[\psi]$ for its class in $\out(G)$.

\subsection{Pure Symmetric Automorphisms}\label{subsection pure symmetric autos}	

\begin{defn}\label{defn pure symmetric autos}
Let $G=G_{1}\ast\dots\ast G_{n}$ be a group which splits as a non-trivial free product, and denote $\mathfrak{S}:=(G_{1},\dots,G_{n})$.
\begin{itemize}
\item We say $\psi\in\aut(G)$ is a \emph{pure symmetric automorphism} of the splitting $G_{1}\ast\dots\ast G_{n}$ if for each $i$ there is some $g_{i}\in G$ such that $\psi(G_{i})=G_{i}^{g_{i}}=g_{i}^{-1}G_{i}g_{i}$. We denote the subgroup of $\aut(G)$ comprising these pure symmetric automorphisms by $\auts(G)$.
\item We say $\Psi\in\out(G)$ is a \emph{pure symmetric outer automorphism} of the splitting if there is some $\psi\in\Psi$ which is a pure symmetric automorphism of the splitting. We denote the subgroup of $\out(G)$ comprising these pure symmetric outer automorphisms by $\outs(G)$.
\item We call a free product $H_{1}\ast\dots\ast H_{n}$ an \emph{$\mathfrak{S}$-free splitting} for $G=G_{1}\ast\dots\ast G_{n}$ if for each $i\in\{1,\dots,n\}$, there exists $g_{i}\in G$ such that $H_{i}=G_{i}^{g_{i}}$ 
and $G_{1}^{g_{1}}\ast\dots\ast G_{n}^{g_{n}}=G$.
\end{itemize}
\end{defn}

\begin{obs}\label{obs Grushko decomposition}
If $G=G_{1}\ast\dots\ast G_{n}$ where the factor groups $G_{i}$ are non-trivial, not infinite cylic, and freely indecomposable, then $G_{1}\ast\dots\ast G_{n}$ is a Grushko decomposition for $G$, and by the Grushko Decomposition Theorem, the $G_{i}$'s are unique up to permutation of their conjugacy classes in $G$.
In particular, if the $G_{i}$'s are additionally pairwise non-isomorphic, we deduce that every automorphism of $G$ is a pure symmetric automorphism of the splitting.
That is, $\auts(G)=\aut(G)$ and $\outs(G)=\out(G)$ (where $\mathfrak{S}=(G_{1},\dots,G_{n})$).
\end{obs}

\begin{rem}
Note that $\inn(G)\le \auts(G) \le \aut(G)$.
We then have that $\faktor{\aut_{\mathfrak{S}}(G)}{\inn(G)}=\outs(G)$ and
moreover, if $\Psi\in\outs(G)$, then \textit{every} representative of $\Psi$ in $\aut(G)$ is in fact in $\auts(G)$.
\end{rem}

\begin{lemma}\label{lemma pure symmetric autos give S-free splittings}
Let $G=G_{1}\ast\dots\ast G_{n}$ and $\mathfrak{S}=(G_{1},\dots,G_{n})$.
Then $H_{1}\ast\dots\ast H_{n}$ is an $\mathfrak{S}$-free splitting for $G$ if and only if
there exists $\psi\in\auts(G)$ such that $H_{i}=\psi(G_{i})$ for each $i\in\{1,\dots,n\}$.
\end{lemma}

\begin{proof}
First, suppose $H_{1}\ast\dots\ast H_{n}$ is an $\mathfrak{S}$-free splitting for $G$.
Then there exist $g_{1},\dots,g_{n}\in G$ such that for each $i$, $H_{i}=G_{i}^{g_{i}}$.
We may now define a map $\psi$ by setting $\psi(g):=g^{g_{i}}$ for each $g\in G_{i}$ and each $i\in\{1,\dots,n\}$.
By the universal property for free products, $\psi$ extends to an endomorphism of $G_{1}\ast\dots\ast G_{n}=G$.
Similarly, the map $\varphi$ given by $\varphi(g):=g^{g_{i}^{-1}}$ for each $g\in G_{i}^{g_{i}}$ and each $i\in\{1,\dots,n\}$ extends to an endomorphism of $G_{1}^{g_{1}}\ast\dots\ast G_{n}^{g_{n}}=G$.
Noting that $\varphi$ is an inverse for $\psi$, we see that $\psi$ is bijective.
Hence $\psi\in\aut(G)$ and thus $\psi\in\auts(G)$.
%

Now suppose $\psi\in\auts(G)$.
Then there exist $g_{1},\dots,g_{n}\in G$ such that for each $i$, $\psi(G_{i})=G_{i}^{g_{i}}$.
Observe that since $\psi$ is an automorphism of $G$ $\psi(G_{1})\ast\dots\ast\psi(G_{n})=\psi(G_{1}\ast\dots\ast G_{n})=G_{1}\ast\dots\ast G_{n}$.
Hence $\psi(G_{1})\ast\dots\ast\psi(G_{n})$ is an $\mathfrak{S}$-free splitting for $G$.
\end{proof}

\subsection{Factor Automorphisms}\label{subsection factor autos}	

\begin{defn}\label{defn factor autos}
Let $G=G_{1}\ast\dots\ast G_{n}$ and $\mathfrak{S}=(G_{1},\dots,G_{n})$.
We say $\varphi\in\auts(G)$ is a \emph{factor automorphism} relative to $\mathfrak{S}$ if for each $i\in\{1,\dots,n\}$, $\varphi|_{G_{i}}$ (i.e. $\varphi$ with domain restricted to the embedding of $G_{i}$ in $G$) is an automorphism of $G_{i}$, that is, $\varphi|_{G_{i}}\in\aut(G_{i})$.

We say $\Phi\in\outs(G)$ is a \emph{factor outer automorphism} relative to $\mathfrak{S}$ if $\Phi$ has a representative $\varphi\in\auts(G)$ which is a factor automorphism relative to $\mathfrak{S}$.

When $\mathfrak{S}$ is understood we will simply say $\varphi$ (or $\Phi$) is a factor (outer) automorphism.
\end{defn}

\begin{rem}
The set of factor automorphisms relative to $\mathfrak{S}$ forms a subgroup of $\auts(G)$ which is isomorphic to $\factorautos$,
via the map $\varphi\mapsto(\varphi|_{G_{1}},\dots,\varphi|_{G_{n}})$.
Note that the only factor automorphism relative to $\mathfrak{S}$ which is also an inner automorphism is the identity, so by the Second Isomorphism Theorem, we also have that the set of factor outer automorphisms relative to $\mathfrak{S}$ in $\outs(G)$ is isomorphic to $\factorautos$.
\end{rem}

\subsection{Whitehead Automorphisms}\label{subsection whitehead autos}	

\begin{defn}\label{defn whitehead autos}
Let $G=G_{1}\ast\dots\ast G_{n}$ and $\mathfrak{S}=(G_{1},\dots,G_{n})$, take $i\in\{1,\dots,n\}$, fix some $x\in G_{i}$, and let $Y\subseteq\{G_{1},\dots,G_{n}\}-\{G_{i}\}$.

We denote by $(Y,x)$ the automorphism in $\auts(G)$ which for each $j\in\{1,\dots,n\}$ maps $g\in G_{j}$ by
$g\mapsto \left\{ \begin{array}{ l l } x^{-1}gx & \text{if \ } G_{j}\in Y \\ g & \text{if \ } G_{j}\not\in Y \end{array} \right.$.
We call such an automorphism a \emph{Whitehead automorphism} relative to $\mathfrak{S}$,
and say that $G_{i}$ is its \emph{operating factor}.

An element $\Psi\in\outs(G)$ will be called a \emph{Whitehead outer automorphism} relative to $\mathfrak{S}$ if it has some representative $\psi\in\auts(G)$ which is a Whitehead automorphism relative to $\mathfrak{S}$.

When $\mathfrak{S}$ is understood we will simply say $\psi$ (or $\Psi$) is a Whitehead (outer) automorphism.
\end{defn}

\begin{rem}
Our notation differs from that of Gilbert \cite{Gilbert1987} in that we do not include the operating factor $G_{i}$ in the set $Y$.
\end{rem}

We write $g^{x}$ for the conjugation $x^{-1}gx$.
If $Y=\{G_{j}\}$ we will often write $(G_{j},x)$ for $(\{G_{j}\},x)$.

\begin{obs}\label{obs outer whitehead has unique whitehead rep}
Note that Whitehead automorphisms relative to $\mathfrak{S}$ leave their operating factor pointwise fixed.
If $\iota_{h}\in\aut(G)$ is the inner automorphism $g\mapsto h^{-1}gh=g^{h}$ and $(\{G_{1},\dots,G_{m}\},g_{0})$ is a Whitehead automorphism relative to $\mathfrak{S}$ with $g_{0}\in G_{i}$ (and $i>m$ for ease of notation), then $(\{G_{1},\dots,G_{m}\},g_{0})\iota_{g}$ cannot pointwise fix any factor $G_{j}$ unless either $g=1$ or $g=g_{0}^{-1}$, in which case $g\in G_{i}$ so $g\not\in G_{j}$. Hence there is no candidate for operating factor unless $g=1$, and so the only Whitehead automorphism relative to $\mathfrak{S}$ in the outer automorphism class $[(\{G_{1},\dots,G_{m}\},g_{0})]$ is $(\{G_{1},\dots,G_{m}\},g_{0})$ itself.
Thus if $\Psi\in\outs(G)$ is a Whitehead outer automorphism relative to $\mathfrak{S}$, then $\Psi$ has a unique representative $\psi\in\auts(G)$ which is a Whitehead automorphism relative to $\mathfrak{S}$.
\end{obs}

We will therefore say that if $\psi\in\auts(G)$ is a Whitehead automorphism relative to $\mathfrak{S}$ with operating factor $G_{i}$, then $[\psi]\in\outs(G)$ has operating factor $G_{i}$.


\section{Graphs of Groups for $G_{1}\ast\dots\ast G_{n}$}\label{section graphs of groups}

We now give a brief background in Bass--Serre theory with the aim of constructing an $\outs(G)$-invariant complex for the splitting $G=G_{1}\ast\dots\ast G_{n}$.
Vertices (0-cells) in this complex will be equivalence classes of certain graphs of groups.
We will require that all of our graphs of groups are trees, and have trivial edge groups.
To aid in our construction, we  define $(n,k)$-trees below, which will satisfy these conditions.
Note that our idea of a graph differs from that of Serre \cite{Serre1980}, in that we have a single unoriented edge between adjacent vertices, which may later be assigned an orientation.

\begin{defn}\label{defn (n,k)-tree}
\hspace{1cm}
\begin{itemize}
\item An \emph{$(n,k)$-tree} $T$ is a finite simplicial tree with vertex set $V(T)=\{v_{1},\dots v_{n}, u_{1}, \dots, u_{k}\}$ (with the possibility that $k=0$, i.e. $V(T)=\{v_{1},\dots,v_{n}\}$) and edge set $E(T)\subseteq V(T)\times V(T)$
such that each vertex $u_{j}$ has valency at least 3.
We consider $(w_{1},w_{2})$ and $(w_{2},w_{1})$ to be the same edge.
We will refer to the $u_{j}$'s as `trivial vertices'.
\item An \emph{orientation} of $T$ is a pair of maps $o:E(T)\to V(T)$ and $t:E(T)\to V(T)$ such that for any edge $e$ of $T$ with endpoints $w_{1}$ and $w_{2}$, we have $\{o(e),t(e)\}=\{w_{1},w_{2}\}$ as sets.
\item An \emph{$(n,k)$-automorphism} of $T$ is a graph automorphism $\gamma$ of $T$ (i.e. a bijection $T\to T$ which sends vertices to vertices and edges to edges, and preserves adjacency)
which satisfies $\gamma(v_{i})\in\{v_{1},\dots,v_{n}\}$ and $\gamma(u_{j})\in\{u_{1},\dots,u_{k}\}$ for each $i\in\{1,\dots,n\}$ and each $j\in\{1,\dots,k\}$.
\end{itemize}
\end{defn}

\subsection{$\mathfrak{S}$-Labellings}	

Now that we have defined our underlying graphs, we can construct our graphs of groups by equipping our graphs with `$\mathfrak{S}$-labellings':

\begin{defn}\label{defn s labelling}
Let $G=G_{1}\ast\dots\ast G_{n}$ be a free product with $\mathfrak{S}=(G_{1},\dots,G_{n})$,
let $H_{1}\ast\dots\ast H_{n}$ be an $\mathfrak{S}$-free splitting for $G$,
and let $\sigma\in S_{n}$ be some permutation of $\{1,\dots,n\}$ (so $\{1,\dots,n\}=\{\sigma(1),\dots,\sigma(n)\}$ as sets).
Let $T$ be an $(n,k)$-tree and $\gamma$ an $(n,k)$-automorphism of $T$.

We define $(\gamma(T):H_{\sigma(1)},\dots,H_{\sigma(n)})$ to be the graph of groups whose edge groups are all trivial, each vertex $\gamma(u_{j})$ for $j\in\{1,\dots,k\}$ has trivial vertex group, and each vertex $\gamma(v_{i})$ for $i\in\{1,\dots,n\}$ has vertex group $H_{\sigma(i)}$.
We call $(\gamma(T):H_{\sigma(1)},\dots,H_{\sigma(n)})$ an \emph{$\mathfrak{S}$-labelling} of $\gamma(T)$.
\end{defn}

\begin{rem}
Since $\gamma$ is an automorphism of $T$, then as abstract graphs we have $\gamma(T)=T$.
Moreover, if $\eta\in S_{n}$ is such that for each $i\in\{1,\dots,n\}$, $\gamma(v_{i})=v_{\eta(i)}$, then $(T:H_{\sigma(\eta^{-1}(1))},\dots,H_{\sigma(\eta^{-1}(n))})$ and $(\gamma(T):H_{\sigma(1)},\dots,H_{\sigma(n)})$ describe identical constructions.
We will thus often suppress $\gamma$ and subsume $\eta$ into the permutation $\sigma$.
\end{rem}

\begin{defn}\label{defn equivalent labellings}
Let $\sigma,\tau\in S_{n}$ and let $T_{H}:=(T:H_{\sigma(1)},\dots,H_{\sigma(n)})$ and $T_{K}:=(T:K_{\tau(1)},\dots,K_{\tau(n)})$ be two $\mathfrak{S}$-labellings of an $(n,k)$-tree $T$.
For $i\in\{1,\dots,n\}$, set $H_{v_{i}}:=H_{\sigma(i)}$ and $K_{v_{i}}:=K_{\tau(i)}$, and for $j\in\{1,\dots,k\}$, set $H_{u_{j}}:=\{1\}$ and $K_{u_{j}}:=\{1\}$.

We say that $T_{H}$ and $T_{K}$ are \emph{equal} as $\mathfrak{S}$-labellings of $T$, written $T_{H}=T_{K}$, if (and only if) the map $v_{j}\mapsto v_{(\sigma^{-1}(\tau(j))}$ extends to an $(n,k)$-automorphism of $T$, and for each $i\in\{1,\dots,n\}$ there exists $h_{i}\in H_{i}$ so that $K_{i}=H_{i}^{h_{i}}$.
In particular, $(T:G_{1}^{h_{1}},\dots,G_{n}^{h_{n}})=(T:G_{\sigma(1)}^{k_{\sigma(1)}},\dots,G_{\sigma(n)}^{k_{\sigma(n)}})$ if and only if the map $v_{j}\mapsto v_{\sigma^{-1}(j)}$ extends to an $(n,k)$-automorphism of $T$, and for each $i\in\{1,\dots,n\}$ there exists $g_{i}\in G_{i}$ such that $k_{i}=g_{i}h_{i}$.

We say that $T_{H}$ and $T_{K}$ are \emph{equivalent} as $\mathfrak{S}$-labellings of $T$, written $T_{H}\simeq T_{K}$, if (and only if)
the map $v_{j}\mapsto v_{(\sigma^{-1}(\tau(j))}$ extends to an $(n,k)$-automorphism of $T$, and there exists an orientation of $T$ such that for every $w\in V(T)$ there exists $h_{w}\in G=G_{1}\ast\dots\ast G_{n}$ with $K_{w}=H_{w}^{h_{w}}$ and moreover $h_{t(e)} {h_{o(e)}}^{-1} \in H_{o(e)}$.
\end{defn}

\begin{rem}
This idea of equivalence is a simplified version of taking isomorphism classes of graphs of groups, as described by Bass in \cite[Section 2]{Bass1993}.
Our simplification relies on having trivial edge groups, and the fact that our graphs of groups are trees.
Note that $\simeq$ is indeed an equivalence relation on the set of $\mathfrak{S}$-labellings of $T$.
We will denote the equivalence class of an $\mathfrak{S}$-labelling $T_{H}$ by $\class{T_{H}}$.
\end{rem}

\begin{obs}\label{obs factor and inner autos with labellings}
Let $T$ be an $(n,k)$-tree and let $H_{1}\ast\dots\ast H_{N}$ be an $\mathfrak{S}$-free splitting for $G=G_{1}\ast\dots\ast G_{n}$ (where $\mathfrak{S}=(G_{1},\dots,G_{n})$).

Let $\varphi\in\auts(G)$ be a factor automorphism.
Then for each $i\in\{1,\dots,n\}$, $\varphi(G_{i})=G_{i}$, and we have $(T:\varphi(G_{\sigma(1}),\dots,\varphi(G_{\sigma(n}))=(T:G_{\sigma(1},\dots,G_{\sigma(n})$.

Now let $x\in G$ and let $\iota_{\sigma(x)}\in\aut(G)$ be the inner automorphism $g\mapsto x^{-1}gx=g^{x}$.
Then $(T:\iota_{x}(H_{\sigma(1)}),\dots,\iota_{x}(H_{\sigma(n)}))=(T:H_{\sigma(1)}^{x},\dots,H_{\sigma(n)}^{x})\simeq(T:H_{\sigma(1)},\dots,H_{\sigma(n)})$
since $x\in G$, and regardlesss of the orientation chosen for $T$, we have $h_{t(e)}h_{o(e)}^{-1}=xx^{-1}=1\in H_{o(e)}$ for each edge $e$ of $T$.

That is to say, inner automorphisms preserve equivalence classes of $\mathfrak{S}$-labellings, and factor automorphisms preserve $\mathfrak{S}$-labellings whose vertex groups are precisely the groups $G_{1},\dots,G_{n}$.
\end{obs}

\subsection{Collapses} 

In order to build our complex, we will need some kind of relation between our graphs of groups; we may then have edges (1-cells) in our complex whenever their endpoints are related.
The relation we use is `collapsing':

\begin{defn}
Let $G=G_{1}\ast\dots\ast G_{n}$ be a free product with $\mathfrak{S}=(G_{1},\dots,G_{n})$
and let $(T:H_{\sigma(1)},\dots,H_{\sigma(n)})$ be an $\mathfrak{S}$-labelling of some $(n,k)$-tree $T$.
\begin{itemize}
\item We say an edge of $(T:H_{\sigma(1)},\dots,H_{\sigma(n)})$ is \emph{collapsible} if it has at least one trivial endpoint (i.e. at least one endpoint whose vertex group is the trivial group).
\item The process of replacing a collapsible edge (including its endpoints) of $(T:H_{\sigma(1)},\dots,H_{\sigma(n)})$ by a single vertex whose vertex group is the free product of the vertex groups of the endpoints of said edge is called \emph{collapsing}.
This new vertex group will always be of the form $H\ast\{1\}$, which we will write simply as $H$ (for $H\in\{\ \{1\}, H_{\sigma(1)},\dots,H_{\sigma(n)} \}$).
\item We say $(T':H_{\tau(1)}',\dots,H_{\tau(n)}')$  is a \emph{collapse} of $(T:H_{\sigma(1)},\dots,H_{\sigma(n)})$ if $(T':H_{\tau(1)}',\dots,H_{\tau(n)}')$ can be achieved as the result of successively collapsing edges of $(T:H_{\sigma(1)},\dots,H_{\sigma(n)})$.
Note that we do not consider a graph of groups to be a collapse of itself.
\end{itemize}
\end{defn}

\begin{obs}\label{obs collapses respect equivalence}
Suppose we have $(T_{1}:H_{\sigma(1)},\dots,H_{\sigma(n)}) \simeq (T_{1}:H'_{\tau(1)},\dots,H'_{\tau(n)})$ for some $(n,k)$-tree $T_{1}$.
Fix some collapsible edge $e$ of $T_{1}$, and let $T_{2}$ be the tree resulting from the collapse of $e$ in $T_{1}$.
Then $T_{2}$ is an $(n,k-1)$-tree and $(T_{2}:H_{\sigma(1)},\dots,H_{\sigma(n)}) \simeq (T_{2}:H'_{\tau(1)},\dots,H'_{\tau(n)})$.

We may thus consider collapses on equivalence classes of $\mathfrak{S}$-labellings in the natural way, i.e. $\class{T'}$ is a collapse of $\class{T}$ if there is some representative $(T:H_{\sigma(1)},\dots,H_{\sigma(n)})$ of $\class{T}$ and some representative $(T':H'_{\tau(1)},\dots,H'_{\tau(n)})$ of $\class{T'}$ such that $(T':H'_{\tau(1)},\dots,H'_{\tau(n)})$ is a collapse of $(T:H_{\sigma(1)},\dots,H_{\sigma(n)})$.
Note that if $\class{(T':H_{\tau(1)}',\dots,H_{\tau(n)}')}$ is a collapse of $\class{(T:H_{\sigma(1)},\dots,H_{\sigma(n)})}$ then $(T':H_{\tau(1)}',\dots,H_{\tau(n)}')\simeq(T':H_{\sigma(1)},\dots,H_{\sigma(n)})$.
\end{obs}


\subsection{Universal Cover of an $\mathfrak{S}$-Labelling} 

The Fundamental Theorem of Bass--Serre Theory tells us that we have a correspondence between graphs of groups and certain $G$-trees (namely, their universal covers), which we will later exploit. We adapt the following construction from Serre \cite[Chapter I Section 5.3]{Serre1980}, restricted to our case of $(n,k)$-trees.

\begin{defn}[Serre {\cite{Serre1980}}]\label{defn universal cover}
Let $G=G_{1}\ast\dots\ast G_{n}$ be a free product with $\mathfrak{S}=(G_{1},\dots,G_{n})$
and let $T_{H}:=(T:H_{1},\dots,H_{n})$ be an $\mathfrak{S}$-labelling of some $(n,k)$-tree $T$ with vertex set $V(T)=\{v_{1},\dots v_{n}, u_{1}, \dots, u_{k}\}$ and edge set $E(T)$.
For $i\in\{1,\dots,n\}$, set $H_{v_{i}}:=H_{i}$ and for $j\in\{1,\dots,k\}$, set $H_{u_{j}}:=\{1\}$.
The \emph{universal cover} of $T_{H}$, denoted $\tree{T_{H}}$, is the infinite tree with vertex set 
$\displaystyle{ V(\tree{T_{H}})=\smashoperator[lr]{ \bigsqcup_{ \substack{ g\in G \\ w\in V(T) } } } H_{w}\coset g }$
and edge set $\displaystyle{ E(\tree{T_{H}})=\smashoperator[lr]{ \bigsqcup_{ \substack{ g\in G \\ e\in E(T) } } } e\coset g }$
(with $o(e\coset g)=H_{o(e)}\coset g$), where $H_{w}\coset g$ is the coset $\{hg \mid h\in H_{w}\}\subset G$.
\end{defn}

\begin{obs}
We have a natural isometric action of $G$ on $\tree{T_{H}}$ given by $(x\coset g) \cdot h = x\coset gh$ for $x\coset g \in V(\tree{T_{H}}) \cup E(\tree{T_{H}})$ and $h\in G$.
Under this action we have that the $G$-stabiliser of each edge in $\tree{T_{H}}$ is trivial (since all edge groups in $T_{H}$ are trivial), and the $G$-stabiliser of a vertex $H_{w}\coset g$ in $\tree{T_{H}}$ is precisely the subgroup $g^{-1} H_{w} g = {H_{w}}^{g}$ of $G$.
\end{obs}

\begin{rem}
If a vertex $w\in V(T)$ has valency $N$ in $T$, then the vertex $H_{w}\coset g \in V(\tree{T_{H}})$ will have valency $N |H_{w}|$ in $\tree{T_{H}}$.
In particular, if $H_{w}=\{1\}$ then these valencies must be equal.
However, if $H_{w}$ is infinite, then $H_{w}\coset g$ will have infinite valency in $\tree{T_{H}}$.
\end{rem}

\subsection{Outer Space for a Free Product}\label{subsection outer space} 

There are several equivalent formulations of an `Outer Space' on which $\outs(G)$ acts.
Originally, the study of the (pure) symmetric automorphisms of a free product by their action on a cellular complex was due to McCullough--Miller\cite{McCullough1996}.
Later, Guirardel--Levitt \cite{Guirardel2007} introduced their deformation space for studying the outer automorphisms of a Grushko decomposition of a free product.
This space retracts onto a cellular complex which, in the case where there is no free rank, exactly matches that of McCullough--Miller.
Since our decomposition need not be a Grushko decomposition, the Outer Space we use is due to McCullough--Miller.
However, the description we give is more in the spirit of Guirardel--Levitt.

Let $G=G_{1}\ast\dots\ast G_{n}$ be a free product with $n\ge3$ where each factor group $G_{i}$ is non-trivial, and set $\mathfrak{S}=(G_{1},\dots,G_{n})$.

\begin{rem}
The relation `$(T_{1}:H_{\sigma(1)},\dots,H_{\sigma(n)})$ is a collapse of $(T_{2}:H_{\sigma(1)},\dots,H_{\sigma(n)})$' defines a strict partial order on the set of collapses of equivalence classes of $\mathfrak{S}$-labellings of a given $(n,k)$-tree. This extends to encompass \textit{all} $(n,k)$-trees, even when $k$ is allowed to vary. 
\end{rem}

\begin{defn}Let $G=G_{1}\ast\dots\ast G_{n}$ and $\mathfrak{S}=(G_{1},\dots,G_{n})$ be as above.
The complex $\mathcal{O}_{n}(\mathfrak{S})$ is the geometric realisation of the poset of equivalence classes of $\mathfrak{S}$-labellings of $(n,k)$-trees (where $n\ge3$ is fixed and $k\ge0$ varies),
with $T_{1}\prec T_{2}$ if (and only if) $T_{2}$ is a collapse of $T_{1}$.
\end{defn}

In this paper, we will always have $\mathfrak{S}=(G_{1},\dots,G_{n})$, and so we will write $\mathcal{O}_{n}$ for $\mathcal{O}_{n}(\mathfrak{S})$.

\begin{rem}
Note that $\mathcal{O}_{n}$ is a simplicial complex.
Moreover for $m\ge2$, for each $(m+1)$-clique in $\mathcal{O}_{n}^{(1)}$ there is a unique $m$-cell in $\mathcal{O}_{n}$ which contains precisely the vertices of that clique.
That is to say, $\mathcal{O}_{n}$ is the flag complex on $\mathcal{O}_{n}^{(1)}$.
\end{rem}

\begin{notation}
We will call cells in $\mathcal{O}_{n}^{(0)}$ `vertices'; these are equivalence classes of $\mathfrak{S}$-labellings of $(n,k)$-trees.
We will call cells in $\mathcal{O}_{n}^{(1)}$ `edges'; there is an edge $[T_{1}]\dash[T_{2}]$ precisely when $[T_{2}]$ is a collapse of $[T_{1}]$.
We may abuse notation and also write $[T_{1}]\dash[T_{2}]$ when $[T_{1}]$ is a collapse of $[T_{2}]$.
\end{notation}

By Observation \ref{obs Grushko decomposition}, if the factor groups $G_{i}$ are additionally freely indecomposable and not infinite cyclic, then $G_{1}\ast\dots\ast G_{n}$ is a Grushko decomposition for $G$, and our description of $\mathcal{O}_{n}$ is precisely the barycentric spine of Guirardel and Levitt's `Outer Space for a Free Product' \cite{Guirardel2007} (where the barycentric spine is a CW-complex resulting from taking the first barycentric subdivision of Outer Space and linearly retracting off the `missing' faces).

We thus have that $\mathcal{O}_{n}$ is isometric to the barycentric spine of Guirdardel and Levitt's Outer Space for a free product of $n$ non-trivial, freely indecomposable, not infinite cyclic, pairwise non-isomorphic factors (even when our factor groups do not satisfy these conditions --- since we want an action of $\outs(G)$, not $\out(G)$).


\begin{defn}\label{defn action of out(G)} 
Let $\Psi\in\outs(G)$ have representative $\psi\in\aut_{\mathfrak{S}}(G)$ and
let $T_{0}=(T:H_{\sigma(1)},\dots,H_{\sigma(n)})$ be a representative of some point $\class{T_{0}}$ in $\mathcal{O}_{n}^{(0)}$.
We define: \[\class{T_{0}}\cdot\Psi := \class{ (T:\psi(H_{\sigma(1)}),\dots,\psi(H_{\sigma(n)})) } \]
\end{defn}

\begin{rem}
Note that by Lemma \ref{lemma pure symmetric autos give S-free splittings}, $\class{T_{0}}\cdot\Psi$ is indeed a point in $\mathcal{O}_{n}$.
One ought to additionally check that this action is independent of both the choice of representative for $\Psi\in\outs{G}$ and the choice of representative for $\class{T_{0}}\in\mathcal{O}_{n}^{(0)}$.
The former point is addressed by Observation \ref{obs factor and inner autos with labellings}
and the latter point is a similarly straightforward exercise.
\end{rem}



\begin{obs}\label{obs action preserves adjacency}
From Observation \ref{obs collapses respect equivalence} we deduce that if $\class{T_{0}}\dash\class{T_{1}}$ is an edge in $\mathcal{O}_{n}$, then so too is $\left(\class{T_{0}}\cdot\Psi\right) \dash \left(\class{T_{1}}\cdot\Psi\right)$ for any $\Psi\in\outs(G)$.
Thus the action of $\outs(G)$ on $\mathcal{O}_{n}^{(0)}$ preserves adjacency, and we may extend this to an action on $\mathcal{O}_{n}$ by sending a cell corresponding to an $(m+1)$-clique $\{V_{0},\dots,V_{m}\}$ to the cell corresponding to the $\outs(G)$-image of the $(m+1)$-clique, $\{V_{0}\cdot\Psi,\dots,V_{m}\cdot\Psi\}$.
\end{obs}

%


\section{A Subcomplex $\mathcal{S}_{n}$ of Outer Space}\label{section Sn}

In this paper, we care about two specific types of vertex in $\mathcal{O}_{n}^{(0)}$: $\alpha$-graph classes and $A$-graph classes.
We will then consider the subcomplex $\mathcal{S}_{n}$ of $\mathcal{O}_{n}$ spanned by these vertices.
This subcomplex will encode Whitehead outer automorphisms of the splitting $G_{1}\ast\dots\ast G_{n}$.

\subsection{$\alpha$-Graphs and $A$-Graphs} 

We begin by describing `$\alpha$-graphs' and `$A$-graphs', and their basic properties (equivalence and adjacency).

\begin{defn}
We define $\alpha$ to be the $(n,1)$-tree in Figure \ref{fig alpha graph shape}.
We say $[\alpha_{0}]\in\mathcal{O}_{n}^{(0)}$ is an $\alpha$-graph class if it has a representative $\alpha_{0}=(\alpha:H_{\sigma(1)},\dots,H_{\sigma(n)})$ which is an $\mathfrak{S}$-labelling of $\alpha$ (for some $\sigma\in S_{n}$).
We call $\alpha_{0}$ itself an $\alpha$-graph.
\end{defn}

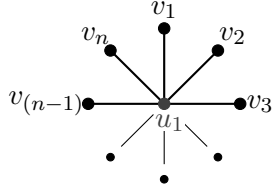
\begin{figure}[h]
\centering
\begin{tikzpicture}  
\draw[thick] (0,0) -- (0,1); 
\draw[thick] (0,0) -- (0.707,0.707); 
\draw[thick] (0,0) -- (1,0); 
\draw[thin] (0.1414,-0.1414) -- (0.5656,-0.5656); 
\draw[thin] (0,-0.2) -- (0,-0.8); 
\draw[thin] (-0.1414,-0.1414) -- (-0.5656,-0.5656); 
\draw[thick] (0,0) -- (-1,0); 
\draw[thick] (0,0) -- (-0.707,0.707); 
\filldraw[darkgray] (0,0) circle [radius=0.075cm];
\filldraw (0,1) circle [radius=0.075cm];
\filldraw (0.707,0.707) circle [radius=0.075cm];
\filldraw (1,0) circle [radius=0.075cm];
\filldraw (0.707,-0.707) circle [radius=0.05cm];
\filldraw (0,-1) circle [radius=0.05cm];
\filldraw (-0.707,-0.707) circle [radius=0.05cm];
\filldraw (-1,0) circle [radius=0.075cm];
\filldraw (-0.707,0.707) circle [radius=0.075cm];
\node[darkgray,fill=white,inner sep=0pt] at (0.075,-0.25) {$u_{1}$};
\node at (0,1.25) {$v_{1}$};
\node at (0.9,0.9) {$v_{2}$};
\node at (1.2625,0) {$v_{3}$};
\node at (-1.55,0) {$v_{(n-1)}$};
\node at (-0.9,0.9) {$v_{n}$};
\end{tikzpicture}
\caption{The Graph `$\alpha$' with vertices $v_{1}$, \dots, $v_{n}$, $u_{1}$}
\label{fig alpha graph shape}
\end{figure}

\begin{obs}\label{obs alpha graph autos}
Given $\sigma\in S_{n}$, we can define a map $\gamma_{\sigma}$ on $\alpha$ by $v_{j}\mapsto v_{\sigma(j)}$, $u_{1}\mapsto u_{1}$, and $u_{1}\dash v_{j} \mapsto u_{1}\dash v_{\sigma(j)}$ for each $j\in\{1,\dots,n\}$.
Clearly $\gamma_{\sigma}$ is a graph automorphism of $\alpha$ and satisfies $\gamma_{\sigma}(v_{j})\in\{v_{1},\dots,v_{n}\}$, that is, $\gamma_{\sigma}$ is an $(n,1)$-automorphism of $\alpha$.
Now for any $\mathfrak{S}$-labelling $(\alpha:H_{\sigma(1)},\dots,H_{\sigma(n)})$ of $\alpha$, we have: $$(\alpha:H_{\sigma(1)},\dots,H_{\sigma(n)})=(\alpha:H_{1},\dots,H_{n})$$
\end{obs}

\begin{lemma}\label{lemma equivalent alpha labellings}
Two $\mathfrak{S}$-labellings, $(\alpha:H_{1},\dots,H_{n})$ and $(\alpha: K_{1},\dots,K_{n})$, of $\alpha$ are equivalent if and only if there exists some $g\in G$ so that for each $j\in\{1,\dots,n\}$, $K_{j}=H_{j}^{g}$.
\end{lemma}

\begin{proof}
Suppose $(\alpha:H_{1},\dots,H_{n})\simeq(\alpha: K_{1},\dots,K_{n})$. Then there exist $h_{u},h_{1},\dots,h_{n}\in G$ so that for all $j\in\{1,\dots,n\}$, $K_{j}=H_{j}^{h_{j}}$ (and $\{1\}^{h_{u}}=\{1\}$). Moreover, for some fixed orientation of $\alpha$, we have $h_{t(e)}h_{o(e)}^{-1}\in H_{o(e)}$ for every edge $e$ of $\alpha$ (where $H_{u_{1}}=\{1\}$).
Every edge of $\alpha$ is of the form $u_{1}\dash v_{j}=:e_{j}$.
If $o(e_{j})=u_{1}$ then $h_{j}h_{u}^{-1}\in\{1\}$, that is, $h_{j}=h_{u}$ and so $K_{j}=H_{j}^{h_{j}}=H_{j}^{h_{u}}$.
If $o(e_{j})=v_{j}$ then $h_{u}h_{j}^{-1}\in H_{j}$ and $K_{j}=H_{j}^{h_{j}}=H_{j}^{h_{u}}$.
Hence taking $g=h_{u}$, we have that for every $j\in\{1,\dots,n\}$, $K_{j}=H_{j}^{g}$.

Now suppose there exists $g\in G$ such that for every $j\in\{1,\dots,n\}$, $K_{j}=H_{j}^{g}$.
Note that $K_{u_{1}}=\{1\}=\{1\}^{g}=H_{u_{1}}^{g}$.
Set $h_{w}=g$ for every vertex $w$ of $\alpha$.
Since $gg^{-1}=1$, then regardless of orientation, we have $h_{t(e)}h_{o(e)}^{-1}=1\in H_{o(e)}$ for every edge $e$ of $\alpha$.
Hence $(\alpha:H_{1},\dots,H_{n})\simeq(\alpha: K_{1},\dots,K_{n})$.
\end{proof}

\begin{defn}
We define $A$ to be the $(n,0)$-tree in Figure \ref{fig A graph shape}.
We say $\class{A_{0}}\in\mathcal{O}_{n}^{(0)}$ is an $A$-graph class if it has a representative $A_{0}=(A:H_{\sigma(1)},\dots,H_{\sigma(n)})$ which is an $\mathfrak{S}$-labelling of $A$ (for some $\sigma\in S_{n}$).
We call $A_{0}$ itself an $A$-graph.
\end{defn}

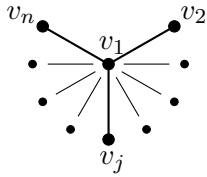
\begin{figure}[h]
\centering
\begin{tikzpicture} 
\draw[thick] (0,0) -- (0,-1);
\draw[thick] (0,0) -- (0.866,0.5);
\draw[thick] (0,0) -- (-0.866,0.5);
\draw (0.2,0) -- (0.8,0);
\draw (0.1732,-0.1) -- (0.6928,-0.4);
\draw (0.1,-0.1732) -- (0.4,-0.6928);
\draw (-0.1,-0.1732) -- (-0.4,-0.6928);
\draw (-0.1732,-0.1) -- (-0.6928,-0.4);
\draw (-0.2,0) -- (-0.8,0);
\filldraw (0,0) circle [radius=0.075]; 
\filldraw (0.866,0.5) circle [radius=0.075]; 
\filldraw (1,0) circle [radius=0.05];
\filldraw (0.866,-0.5) circle [radius=0.05];
\filldraw (0.5,-0.866) circle [radius=0.05];
\filldraw (0,-1) circle [radius=0.075]; 
\filldraw (-0.5,-0.866) circle [radius=0.05];
\filldraw (-0.866,-0.5) circle [radius=0.05];
\filldraw (-1,0) circle [radius=0.05];
\filldraw (-0.866,0.5) circle [radius=0.075]; 
\node at (0.05,0.25) {$v_{1}$};
\node at (0.0525,-1.3) {$v_{j}$};
\node at (1.125,0.65) {$v_{2}$};
\node at (-1.15,0.65) {$v_{n}$};
\end{tikzpicture}
\caption{The Graph `$A$' with Vertices $v_{1}$, $v_{2}$, \dots, $v_{n}$}
\label{fig A graph shape}
\end{figure}

\begin{obs}\label{obs A graph autos}
Note that if $\eta\in S_{n}$ fixes $1$ (i.e. $\eta(1)=1$), then $\eta$ extends to a graph automorphism of $A$ via the same method as Observation \ref{obs alpha graph autos}.
Now fix $\sigma\in S_{n}$ with $\sigma(1)=i$ for some $i\in\{1,\dots,n\}$ and let $(A:H_{\sigma(1)},\dots,H_{\sigma(n)})$ be an $\mathfrak{S}$-labelling of $A$.
Taking $\tau:=(i \ i-1 \ \dots \ 2 \ 1)\in S_{n}$, we have that $\sigma^{-1}\tau\in S_{n}$ and $\sigma^{-1}(\tau(1))=\sigma^{-1}(i)=1$.
Thus $(A:H_{\sigma(1)},H_{\sigma(2)},\dots,H_{\sigma(n)})=(A:H_{\tau(1)},H_{\tau(2)},\dots,H_{\tau(i)},H_{\tau(i+1)},\dots,H_{\tau(n)})=(A:H_{i},H_{1},\dots,H_{i-1},H_{i+1},\dots,H_{n})$.
\end{obs}

\begin{notation}
We will write $(A:H_{i};H_{1},\dots,H_{n})$ for the $\mathfrak{S}$-labelling
\\ \noindent
$(A:H_{i},H_{1},\dots,H_{i-1},H_{i+1},\dots,H_{n})$ of $A$.
Note that for any $\sigma\in S_{n}$ with $\sigma(i)=i$, we have: $$(A:H_{i};H_{1},\dots,H_{n})=(A:H_{i};H_{\sigma(1)},\dots,H_{\sigma(n)})$$
\end{notation}

\begin{lemma}\label{lemma equivalent A labellings}
Two $\mathfrak{S}$-labellings, $(A:H_{i};H_{{1}},\dots,H_{{n}})$ and $(A:K_{k};K_{{1}},\dots,K_{{n}})$, of $A$ are equivalent if and only if $k=i$ and there exist $g\in G$ and $g_{1},\dots,g_{n}\in H_{i}$ so that for each $j\in\{1,\dots,n\}$, $K_{j}=H_{j}^{g_{j}g}$.
\end{lemma}

\begin{proof}
Suppose $(A:H_{i};H_{{1}},\dots,H_{{n}})\simeq(A:K_{k};K_{{1}},\dots,K_{{n}})$.
Then there exist $h_{1},\dots,h_{n}\in G$ so that for all $j\in\{1,\dots,n\}$, $K_{j}=H_{j}^{h_{j}}$.
Moreover, for some fixed orientation of $A$, we have $h_{t(e)}h_{o(e)}^{-1}\in H_{o(e)}$ for every edge $e$ of $A$.
Since any $(n,0)$-automorphism of $A$ must fix the vertex $v_{1}$, we deduce that $k=i$.
To ease notation, we will relabel the vertices of $A$ by 
$w_{j}:= \left\{ \begin{array}{ l l } v_{j+1} & \text{if \ } j<i \\ v_{1} & \text{if \ } j=i \\ v_{j} & \text{if \ } j>i \end{array} \right.$.
Now each edge of $A$ is of the form $w_{i}\dash w_{j}=:e_{j}$ for $j\in\{1,\dots,n\}-\{i\}$.
If $o(e_{j})=w_{j}$ then $h_{i}h_{j}^{-1}\in H_{j}$ and we have $K_{j}=H_{j}^{h_{j}}=H_{j}^{h_{i}}$.
If $o(e_{j})=w_{i}$ then $h_{j}h_{i}^{-1}\in H_{i}$. Note that $K_{j}=H_{j}^{h_{j}}=H_{j}^{h_{j}h_{i}^{-1}h_{i}}$.
Now taking $g:=h_{i}$, $g_{i}:=1$, and for $j\in\{1,\dots,n\}-\{i\}$,
$g_{j}:= \left\{ \begin{array}{ l l } 1 & \text{if \ } o(e_{j})=w_{j} \\ h_{j}h_{i}^{-1} & \text{if \ } o(e_{j})=w_{i} \end{array} \right.$,
we have that for each $j\in\{1,\dots,n\}$, $K_{j}=H_{j}^{g_{j}g}$, with $g_{1},\dots,g_{n}\in H_{i}$.

Now suppose there exist $g\in G$ and $g_{1},\dots,g_{n}\in H_{i}$ such that for every $j\in\{1,\dots,n\}$, $K_{j}=H_{j}^{g_{j}g}$.
We choose an orientation of $A$ by setting $o(e_{j}):=w_{i}$ (where $e_{j}:= w_{i}\dash w_{j}$) for each $j\in\{1,\dots,n\}-\{i\}$, and we set $h_{j}:=g_{j}g$ for each $j\in\{1,\dots,n\}$.
Now for each $j\in\{1,\dots,n\}-\{i\}$, we have $h_{t(e)}h_{o(e)}^{-1}=(g_{j}g)(g_{i}g)^{-1}=g_{j}gg^{-1}g_{i}^{-1}=g_{j}g_{i}^{-1}\in H_{i}$ since $g_{i}, g_{j}\in H_{i}$.
Thus $(A:H_{i};H_{{1}},\dots,H_{{n}})\simeq(A:K_{i};K_{{1}},\dots,K_{{n}})$.
\end{proof}


\begin{lemma}\label{lemma alphas collapse to As}
Let $\class{\alpha_{0}}\in\mathcal{O}_{n}^{(0)}$ be an $\alpha$-graph class with representative $(\alpha:H_{1},\dots,H_{n})$.
Then $\class{T}\in\mathcal{O}_{n}^{(0)}$ is a collapse of $\class{\alpha_{0}}$ if and only $\class{T}$ is an $A$-graph class with representative $(A:H_{i};H_{1},\dots,H_{n})$ for some $i\in\{1,\dots,n\}$.
\end{lemma}

\begin{proof}
First, note that $\alpha_{0}$ has only one trivial vertex, so any collapse of $\alpha_{0}$ is the result of collapsing precisely one edge in $\alpha_{0}$.
For $i\in\{1,\dots,n\}$, denote the edge in $\alpha_{0}$ between $\{1\}$ and $H_{i}$ by $e_{i}$.
Observe that every edge of $\alpha_{0}$ is of this form, and so every edge of $\alpha_{0}$ is collapsible.
Collapsing an edge $e_{i}$ of $\alpha_{0}$ results in the graph of groups $(A:H_{i};H_{1},\dots,H_{n}\}$.
Hence $T$ is a collapse of $\alpha_{0}$ if and only if $T$ results from collapsing some edge $e_{i}$ of $\alpha_{0}$,
so $T$ is a collapse of $\alpha_{0}$ if and only if $T=(A:H_{i};H_{1},\dots,H_{n})$ for some $i\in\{1,\dots,n\}$.
\end{proof}

\subsection{Construction of the Subcomplex $\mathcal{S}_{n}$} 

We are now ready to build our subcomplex $\mathcal{S}_{n}$ of $\mathcal{O}_{n}$.
We will then see how we can move from $\alpha$-graph class to $\alpha$-graph class in $\mathcal{S}_{n}$ by applying elements of $\outs(G)$ found in the $\outs(G)$-stabilisers of $A$-graph classes.

\begin{defn}\label{defn Sn}
We define $\mathcal{S}_{n}$ to be the subcomplex of $\mathcal{O}_{n}$ comprising $\alpha$-graph classes and $A$-graph classes, as well as all edges whose endpoints are $\alpha$-graph classes or $A$-graph classes.
That is, $\mathcal{S}_{n}$ is the subcomplex of $\mathcal{O}_{n}$ spanned by $\alpha$-graph classes and $A$-graph classes.
\end{defn}

\begin{obs}\label{obs Sn properties}
By Lemma \ref{lemma alphas collapse to As},
we see that $\mathcal{S}_{n}$ is a 1-dimensional simplicial complex (i.e. a graph).
Moreover, since $\alpha$ is an $(n,1)$-tree and $A$ is an $(n,0)$-tree and $(n,k)$-trees collapse to $(n,k-1)$-trees, then no $\alpha$-graph can collapse to another $\alpha$-graph, and no $A$-graph can collapse to another $A$-graph.
Thus $\mathcal{S}_{n}$ is bipartite between the set of $\alpha$-graph classes and the set of $A$-graph classes.
\end{obs}

\begin{notation}
We denote $\funalpha:=(\alpha:G_{1},\dots,G_{n})$ and for $i\in\{1,\dots,n\}$, $\funA{i}:=(A:G_{i};G_{1},\dots,G_{n})=(A:G_{i},G_{1},\dots,G_{i-1},G_{i+1},\dots,G_{n})$.
\end{notation}

\begin{obs}\label{obs collapses of funalpha}
From Lemma \ref{lemma alphas collapse to As} we see that the collapses of $\class{\funalpha}$ are precisely the vertices $\class{\funA{1}},\dots,\class{\funA{n}}$ in $\mathcal{S}_{n}^{(0)}$.
\end{obs}

\begin{obs}\label{obs fun dom for Sn}
The subcomplex $\mathcal{S}_{n}$ inherits the $\outs(G)$ action from $\mathcal{O}_{n}$.
Note that by Lemma \ref{lemma pure symmetric autos give S-free splittings}, the subgraph of $\mathcal{S}_{n}$ comprising the $\alpha$-graph class $\class{\funalpha}$, the $A$-graph classes $\funA{1}$, \dots, $\funA{n}$, and the edges joining them forms a strict fundamental domain for the action of $\outs(G)$ restricted to $\mathcal{S}_{n}$.
In particular, there are precisely $n+1$ $\outs(G)$-orbits of vertices and $n$ $\outs(G)$-orbits of edges in $\mathcal{S}_{n}$.
\end{obs}

\begin{lemma}\label{lemma alphas differ by psi in stab A}
Suppose we have an edge path $\class{\alpha_{0}} \dash \class{A_{1}} \dash \class{\alpha_{1}}$ in $\mathcal{S}_{n}$, where $\class{\alpha_{0}}$ and $\class{\alpha_{1}}$ are $\alpha$-graph classes and $\class{A_{1}}$ is an $A$-graph class.
Then there exists $\Psi\in\outs(G)$ with $\class{A_{1}}\cdot\Psi=\class{A_{1}}$ and $\class{\alpha_{1}}=\class{\alpha_{0}}\cdot\Psi$.
\end{lemma}


\begin{proof}
Suppose $\class{\alpha_{0}}$ has representative $\alpha_{0}=(\alpha:H_{1},\dots,H_{n})$ and $\class{\alpha_{1}}$ has representative $\alpha_{1}=(\alpha:K_{1},\dots,K_{n})$.
Then $H_{1}\ast\dots\ast H_{n}$ and $K_{1}\ast\dots\ast K_{n}$ are both $\mathfrak{S}$-free splittings for $G$, so by Lemma \ref{lemma pure symmetric autos give S-free splittings}, there exist $\varphi,\psi\i\auts(G)$ so that for each $j\in\{1,\dots,n\}$, $H_{j}=\varphi(G_{j})$ and $K_{j}=\psi(G_{j})$ (hence $K_{j}=\psi(\varphi^{-1}(H_{j}))$).
In particular, $\class{\alpha_{1}}=\class{\alpha_{0}}\cdot[\varphi^{-1}\psi]$ where $[\varphi^{-1}\psi]\in\outs(G)$ is the outer automorphism class containing $\varphi^{-1}\psi$.
Now by Observation \ref{obs collapses respect equivalence}, since $\class{A_{1}}$ is a collapse of both $\class{\alpha_{0}}$ and $\class{\alpha_{1}}$,
we have $A_{1}\simeq(A:H_{i};H_{1},\dots,H_{n})\simeq(A:K_{i};K_{1},\dots,K_{n})$ for some $i\in\{1,\dots,n\}$.
Hence
\begin{align*}
\class{A_{1}} \cdot [\varphi^{-1}\psi]	&	=	\class{ ( A:H_{i};H_{1},\dots,H_{n} ) } \cdot [\varphi^{-1}\psi]						\\
						&	=	\class{ ( A:\psi(\varphi^{-1}(H_{i}));\psi(\varphi^{-1}(H_{1})),\dots,\psi(\varphi^{-1}(H_{n})) ) }	\\
						&	=	\class{ ( A:K_{i};K_{1},\dots,K_{n} ) }										\\
						&	=	\class{A_{1}}
\end{align*}
\end{proof}


\section{Connectedness of the Subcomplex $\mathcal{S}_{n}$}\label{section connectedness}

Recall that $G=G_{1}\ast\dots\ast G_{n}$, $\mathfrak{S}=(G_{1},\dots,G_{n})$, and $\funalpha=(\alpha:G_{1},\dots,G_{n})$.
The main goal of this Section is to prove Proposition \ref{prop alpha-A paths} --- that is, given an arbitrary $\alpha$-graph class $\class{\alpha_{0}}$ in $\mathcal{S}_{n}$, we can find a path in $\mathcal{S}_{n}$ from $\class{\alpha_{0}}$ to $\class{\funalpha}$.
An immediate corollary of this is that $\mathcal{S}_{n}$ is path connected (Corollary \ref{cor Sn path connected}).
Proposition \ref{prop alpha-A paths} will be key to proving our main theorem, Theorem \ref{main theorem}.

\subsection{The Universal Cover of $\funalpha$}

The majority of the arguments in this Section will take place in the universal cover of $\funalpha$.
We will thus begin by establishing some notation regarding this, as well as some useful (but technical) lemmas.

\begin{notation}
Let $U:=\{1\}$ be the vertex group in $\funalpha$ for the vertex $u_{1}$ of the graph $\alpha$ (see Figure \ref{fig alpha graph shape}).
From Definition \ref{defn universal cover}, we see that the universal cover of $\funalpha$, denoted $\tree{\funalpha}$, has vertex set
$\displaystyle{ V(\tree{\funalpha})=\smashoperator[lr]{ \bigsqcup_{ \substack{ g\in G \\ i\in\{1,\dots,n\} } } } G_{i} \coset g \ \sqcup \smashoperator[r]{ \bigsqcup_{ g\in G } } U \coset g }$
and edge set
$\displaystyle{ E(\tree{\funalpha})=\smashoperator[lr]{ \bigsqcup_{ \substack{ g\in G \\ j\in\{1,\dots,n\} } } } e_{j} \coset g }$
where for each $j\in\{1,\dots,n\}$, $e_{j}$ is the edge joining $\{1\}=U$ and $G_{j}$ in $\funalpha$.
Note then that $e_{j}\coset g$ has endpoints $U\coset g$ and $G_{j}\coset g$.
We will call vertices of the form $G_{i}\coset g$ `$C$-vertices' (`$C$' for coset) and vertices of the form $U\coset g$ `$U$-vertices'.
Given $H\in\{ U,G_{1},\dots,G_{n}\}$, we will often write $H\coset 1=H$.
\end{notation}

\begin{obs}\label{obs tree of funalpha properties}
Note that $\tree{\funalpha}$ is bipartite between the set of $C$-vertices and the set of $U$-vertices,
that is, every neighbour of a $U$-vertex is a $C$-vertex, and vice versa.
Every edge of $\tree{\funalpha}$ has trivial $G$-stabiliser, and so does every $U$-vertex.
A $C$-vertex $G_{i}\coset g$ of $\tree{\funalpha}$ has stabiliser $G_{i}^{g}$ in $G$.
Note that every $U$-vertex has valency exactly $n$; in particular the neighbours of a vertex $U\coset g$ of $\tree{\funalpha}$ are precisely the vertices $G_{1}\coset g, \dots, G_{n}\coset g$.
A $C$-vertex $G_{i}\coset g$ has valency equal to $|G_{i}|$ --- this will often be infinite.
\end{obs}

We sketch part of the universal cover $\tree{\funalpha}$ of $\funalpha$ in Figure \ref{fig tree of fun alpha}.

\begin{figure}[h]
\centering
\begin{tikzpicture}
[ fundom/.pic={
\draw[thick,->-] (0,0) -- (0.259,0.966);
\draw[thin] (0.153,0.129) -- (0.613,0.514);
\draw[thin] (0.199,0.018) -- (0.797,0.070);
\draw[thick,->-] (0,0) -- (0.866,-0.5);
\draw[thin] (0.068,-0.188) -- (0.274,-0.752);
\draw[thin] (-0.068,-0.188) -- (-0.274,-0.752);
\draw[thick,->-] (0,0) -- (-0.866,-0.5);
\draw[thin] (-0.199,0.018) -- (-0.797,0.070);
\draw[thin] (-0.153,0.129) -- (-0.613,0.514);
\draw[thick,->-] (0,0) -- (-0.259,0.966);
\filldraw[Green] (0,0) circle [radius=0.075cm];
\filldraw[black] (0.259,0.966) circle [radius=0.075cm];
\filldraw[black] (0.689,0.579) circle [radius=0.05cm];
\filldraw[black] (0.831,0.344) circle [radius=0.05cm];
\filldraw[black] (0.897,0.078) circle [radius=0.05cm];
\filldraw[Blue] (0.866,-0.5) circle [radius=0.075cm];
\filldraw[black] (0.308,-0.846) circle [radius=0.05cm];
\filldraw[black] (0,-0.9) circle [radius=0.05cm];
\filldraw[black] (-0.308,-0.846) circle [radius=0.05cm];
\filldraw[Red] (-0.866,-0.5) circle [radius=0.075cm];
\filldraw[black] (-0.897,0.078) circle [radius=0.05cm];
\filldraw[black] (-0.831,0.344) circle [radius=0.05cm];
\filldraw[black] (-0.689,0.579) circle [radius=0.05cm];
\filldraw[black] (-0.259,0.966) circle [radius=0.075cm];
}]
\draw (0,0) pic[rotate=0] {fundom}; 
\draw (1.866,-0.5) pic[rotate=210] {fundom}; 
\draw (1.866,1.232) pic[rotate=90] {fundom}; 
\filldraw[white] (2.766,1.232) circle [radius=0.06cm];
\draw (3.366,0.366) pic[rotate=330] {fundom}; 
\filldraw[white] (4.266,0.25) circle [radius=0.06cm];
\draw (-0.866,-1.5) pic[rotate=240] {fundom}; 
\draw (-2.232,-2.866) pic[rotate=90] {fundom}; 
\filldraw[white] (-1.332,-2.866) circle [radius=0.06cm];
\draw (-1.732,0) pic[rotate=120] {fundom}; 
\draw (-2.232,1.866) pic[rotate=330] {fundom}; 
\filldraw[white] (-1.332,1.75) circle [radius=0.06cm];
\node[Green,fill=white,inner sep=0pt] at (0,-0.3) {\small{$U$}}; 							
\node[black] at (0.35,1.2) {\small{$G_{1}$}};
\node[Blue] at (0.866,-0.25) {\small{$G_{i}$}};
\node[Red] at (-0.866,-0.225) {\small{$G_{j}$}};
\node[black] at (-0.25,1.2) {\small{$G_{n}$}};
\node[Green,fill=white,inner sep=0pt] at (2.275,-0.45) {\small{$U\coset\textcolor{Blue}{x}$}}; 				
\node[black] at (2.1,-1.7) {\small{$G_{1}\coset\textcolor{Blue}{x}$}};
\node[Red] at (2.65,0.55) {\small{$G_{j}\coset\textcolor{Blue}{x}$}};
\node[black] at (2.95,-1.45) {\small{$G_{n}\coset\textcolor{Blue}{x}$}};
\node[Green,fill=white,inner sep=0pt] at (2.375,1.25) {\small{$U\coset\textcolor{Red}{y}\textcolor{Blue}{x}$}};	 	
\node[black] at (0.925,1.75) {\small{$G_{1}\coset\textcolor{Red}{y}\textcolor{Blue}{x}$}};
\node[Blue] at (2.85,2.25) {\small{$G_{i}\coset\textcolor{Red}{y}\textcolor{Blue}{x}$}};
\node[black] at (0.925,0.8) {\small{$G_{n}\coset\textcolor{Red}{y}\textcolor{Blue}{x}$}};
\node[Green,fill=white,inner sep=0pt] at (3.875,0.325) {\small{$U\coset\textcolor{Red}{z}\textcolor{Blue}{x}$}};	 	
\node[black] at (4.6,1.2) {\small{$G_{1}\coset\textcolor{Red}{z}\textcolor{Blue}{x}$}};
\node[Blue] at (4.25,-0.7625) {\small{$G_{i}\coset\textcolor{Red}{z}\textcolor{Blue}{x}$}};
\node[black] at (3.9,1.6) {\small{$G_{n}\coset\textcolor{Red}{z}\textcolor{Blue}{x}$}};
\node[Green,fill=white,inner sep=0pt] at (-1.25,-1.35) {\small{$U\coset\textcolor{Red}{y}$}}; 				
\node[black] at (0.25,-2.4) {\small{$G_{1}\coset\textcolor{Red}{y}$}};
\node[Blue] at (-1.85,-1.7) {\small{$G_{i}\coset\textcolor{Red}{y}$}};
\node[black] at (0.575,-1.9) {\small{$G_{n}\coset\textcolor{Red}{y}$}};
\node[Green,fill=white,inner sep=0pt] at (-1.725,-2.85) {\small{$U\coset\textcolor{Blue}{x}\textcolor{Red}{y}$}}; 	
\node[black] at (-3.75,-2.55) {\small{$G_{1}\coset\textcolor{Blue}{x}\textcolor{Red}{y}$}};
\node[Red] at (-1.3,-4) {\small{$G_{j}\coset\textcolor{Blue}{x}\textcolor{Red}{y}$}};
\node[black] at (-3.775,-3.1) {\small{$G_{n}\coset\textcolor{Blue}{x}\textcolor{Red}{y}$}};
\node[Green,fill=white,inner sep=0pt] at (-2.1,0.15) {\small{$U\coset\textcolor{Red}{z}$}}; 				
\node[black] at (-3.175,-0.25) {\small{$G_{1}\coset\textcolor{Red}{z}$}};
\node[Blue] at (-1.3,1.05) {\small{$G_{i}\coset\textcolor{Red}{z}$}};
\node[black] at (-2.475,-0.9) {\small{$G_{n}\coset\textcolor{Red}{z}$}};
\node[Green,fill=white,inner sep=0pt] at (-1.725,1.825) {\small{$U\coset\textcolor{Blue}{x}\textcolor{Red}{z}$}}; 	
\node[black] at (-1.025,2.725) {\small{$G_{1}\coset\textcolor{Blue}{x}\textcolor{Red}{z}$}};
\node[Red] at (-3.775,1.85) {\small{$G_{j}\coset\textcolor{Blue}{x}\textcolor{Red}{z}$}};
\node[black] at (-1.6,3.05) {\small{$G_{n}\coset\textcolor{Blue}{x}\textcolor{Red}{z}$}};
\end{tikzpicture}
\caption{A Small Part of $\tree{\funalpha}$, with $x\in G_{i}$ and $y,z\in G_{j}$}
\label{fig tree of fun alpha}
\end{figure}

\begin{obs}\label{obs malnormalcy and coset equivalence}
Fix $i\in\{1,\dots,n\}$ and let $g_{i},h\in G=G_{1}\ast\dots\ast G_{n}$.
As cosets we have that $G_{i}\coset g_{i}h=G_{i}\coset g_{i} \Leftrightarrow G_{i}\coset g_{i}hg_{i}^{-1} \Leftrightarrow g_{i}hg_{i}^{-1}\in G_{i} \Leftrightarrow h\in g_{i}^{-1} G_{i} g_{i}=G_{i}^{g_{i}}$.
Moreover, since $G_{i}$ is malnormal in $G$ (as it is a free factor), we have that $G_{i}^{g_{i}h}=G_{i}^{g_{i}} \Leftrightarrow G_{i}^{g_{i}hg_{i}^{-1}}=G_{i} \Leftrightarrow g_{i}hg_{i}^{-1}\in G_{i} \Leftrightarrow h\in g_{i}^{-1} G_{i} g_{i}=G_{i}^{g_{i}}$.
Thus $G_{i}\coset g_{i}h=G_{i}\coset g_{i}$ as cosets if and only if $G_{i}^{g_{i}h}=G_{i}^{g_{i}}$ as subgroups of $G$.
\end{obs}

\begin{notation}
We will write $[P,Q]$ for the geodesic in $\tree{\funalpha}$ from the vertex $P$ to the vertex $Q$.
We write $[P_{1},P_{2},\dots,P_{q}]$ to mean that the vertices $P_{2}, \dots, P_{q-1}$ lie on the geodesic $[P_{1},P_{q}]$, in the order given. We allow the possibility that $P_{i}=P_{i+1}$ for $i\in\{1,\dots,q-1\}$.
The length of a geodesic $[P,Q]$ is simply the number of edges it contains (equivalently, one less than the number of vertices it contains), written $|[P,Q]|$.
\end{notation}

\begin{lemma}\label{lemma geodesics in tree of funalpha}
Let $G_{j}\coset g_{j}$ be a $C$-vertex in $\tree{\funalpha}$ (so $j\in\{1,\dots,n\}$ and $g_{j}\in G$) and suppose $h\in G_{k}^{g_{k}}$ is non-trivial, for some $k\in\{1,\dots,n\}$ and some $g_{k}\in G$ with $G_{k}^{g_{k}} \ne G_{j}^{g_{j}}$.
Then the $C$-vertex $G_{k}\coset g_{k}$ in $\tree{\funalpha}$ lies halfway along the $\tree{\funalpha}$-geodesic $[G_{j}\coset g_{j}, G_{j}\coset g_{j}h]$.
\end{lemma}

\begin{rem}
This Lemma holds more generally whenever we have a tree equipped with an edge-free action and an elliptic element.
Nevertheless, it will be useful to familiarise ourselves with the details of the tree $\tree{\funalpha}$ by proving it in this specific case.
\end{rem}

\begin{proof}
We consider the $\tree{\funalpha}$-geodesic $[G_{j}\coset g_{j},G_{k}\coset g_{k}]$.
Since $G_{k}^{g_{k}} \ne G_{j}^{g_{j}}$ then by Observation \ref{obs malnormalcy and coset equivalence}, $G_{j}\coset g_{j}\ne G_{k}\coset g_{k}$, and are thus distinct $C$-vertices in $\tree{\funalpha}$.
Due to the bipartite structure of $\tree{\funalpha}$, there is some  $U$-vertex $U\coset x$ (where $x\in G$) on the geodesic $[G_{j}\coset g_{j},G_{k}\coset g_{k}]$.
Let $U\coset y$ be the closest such vertex to $G_{k}\coset g_{k}$ (i.e. its neighbour).

Note that since $h\in G_{k}^{g_{k}}$ then $h$ stabilises $G_{k}\coset g_{k}$, that is, $G_{k}\coset g_{k}\cdot h = G_{k}\coset g_{k}h= G_{k}\coset g_{k}$.
By the action of $G$ on $\tree{\funalpha}$, we have $[G_{j}\coset g_{j},G_{k}\coset g_{k}]\cdot h = \left[{G_{j}\coset g_{j}}\cdot h,{G_{k}\coset g_{k}}\cdot h\right] = [G_{j}\coset g_{j}h,G_{k}\coset g_{k}]$.
Moreover, the vertex $U\coset yh$ in $\tree{\funalpha}$ lies on $[G_{j}\coset g_{j}h,G_{k}\coset g_{k}]$ and is adjacent to $G_{k}\coset g_{k}$.
Since each $U$-vertex has trivial stabiliser (and $h$ is assumed to be non-trivial) then $U\coset y$ and $U\coset yh$ are distinct vertices in $\tree{\funalpha}$.
Thus $[G_{j}\coset g_{j},G_{k}\coset g_{k}] \cap [G_{j}\coset g_{j}h,G_{k}\coset g_{k}] = \{ G_{k}\coset g_{k} \}$.

We conclude that the $\tree{\funalpha}$-geodesic $[G_{j}\coset g_{j}, G_{j}\coset g_{j}h]$ is the concatenation of the geodesics $[G_{j}\coset g_{j},G_{k}\coset g_{k}]$ and $[G_{j}\coset g_{j},G_{k}\coset g_{k}]\cdot h=[G_{k}\coset g_{k},G_{j}\coset g_{j}h]$.
Hence $G_{k}\coset g_{k}$ lies halfway along $[G_{j}\coset g_{j}, G_{j}\coset g_{j}h]$, as required.
\end{proof}

The following Lemma is a technical result required for the proof of Lemma \ref{lemma spokes fold} (which in turn is required for our main ingredient of Proposition \ref{prop alpha-A paths}, Lemma \ref{lemma volume is decreasable}).

\begin{lemma}\label{lemma P(a,b)}
Suppose $G_{1}^{g_{1}}\ast\dots\ast G_{n}^{g_{n}}$ is an $\mathfrak{S}$-free splitting of $G$.
Let $G_{k_{0}}\coset g_{k_{0}}h_{1}\dots h_{m}$ be a vertex in $\tree{\funalpha}$ where for each $i\in\{1,\dots,m\}$, $h_{i}$ belongs to a factor group $G_{k_{i}}^{g_{k_{i}}}$ such that $k_{i}\ne k_{i-1}$ (that is, $h_{1}\dots h_{m}$ is a reduced word with respect to the splitting $G_{1}^{g_{1}}\ast\dots\ast G_{n}^{g_{n}}$, and $h_{1}\not\in G_{k_{0}}^{g_{k_{0}}}$).
Then for some $j,l\in \{0,\dots,m\}$ with $k_{j}\ne k_{l}$ we have that the vertex $G_{k_{j}}\coset g_{k_{j}}$ in $\tree{\funalpha}$ lies on the $\tree{\funalpha}$-geodesic $[G_{k_{0}}\coset g_{k_{0}}h_{1}\dots h_{m}, G_{k_{l}}\coset g_{k_{l}}]$, that is, we have $[G_{k_{0}}\coset g_{k_{0}}h_{1}\dots h_{m}, G_{k_{j}}\coset g_{k_{j}}, G_{k_{l}}\coset g_{k_{l}}]$.
\end{lemma}

\begin{proof}
For bookkeeping purposes, we set $h_{0}=1$ (so $g_{k_{0}}h_{1}\dots h_{m}=g_{k_{0}}q_{m}$ and $g_{k_{0}}h_{0}=g_{k_{0}}$),
and for brevity, we will write $C_{k_{j}}$ for the coset $G_{k_{j}}\coset g_{k_{j}}$
and $q_{j}$ for  the element $h_{0}\dots h_{j}$. Note then that $q_{j+1}=q_{j}h_{j+1}$.

Suppose for contradiction that the Lemma never holds, that is, we never have
\\ \noindent
$[C_{k_{0}}q_{m}, C_{k_{j}}, C_{k_{l}}]$ for $j,l\in\{0,\dots,m\}$ with $k_{j}\ne k_{l}$.

We claim that we must then have that for each $i\in\{1,\dots,m\}$, the vertex $C_{k_{m-(i-1)}}$ in $\tree{\funalpha}$ lies on the $\tree{\funalpha}$-geodesic $[C_{k_{0}}q_{m},C_{k_{0}}q_{m-i}]$, that is, we have 
$[C_{k_{0}}q_{m}, C_{k_{m-(i-1)}}, C_{k_{0}}q_{m-i}]$.
We proceed by induction on $i$.

Since $q_{m}=q_{m-1}h_{m}$, then by Lemma \ref{lemma geodesics in tree of funalpha}, we have $[C_{k_{0}}q_{m}, C_{k_{m}}, C_{k_{0}}q_{m-1}]$. Thus the claim holds for $i=1$.

Now take $1\le i\le m-1$ and suppose we have $[C_{k_{0}}q_{m}, C_{k_{m-(i-1)}}, C_{k_{0}}q_{m-i}]$.
By Lemma \ref{lemma geodesics in tree of funalpha}, we have $[C_{k_{0}}q_{m-(i+1)}h_{m-i}, C_{k_{m-i}}, C_{k_{0}}q_{m-(i+1)}]$, with $q_{m-(i+1)}h_{m-i}=q_{m-i}$.
Take $V_{i}$ to be the vertex in $\tree{\funalpha}$ satisfying 
$[C_{k_{0}}q_{m-i}, C_{k_{0}}q_{m}] \cap [C_{k_{0}}q_{m-i}, C_{k_{0}}q_{m-(i+1)}] = [C_{k_{0}}q_{m-i},V_{i}]$.
Note that $[C_{k_{0}}q_{m}, C_{k_{m-(i-1)}}, V_{i}]$ would imply $[C_{k_{0}}q_{m}, C_{k_{m-(i-1)}}, C_{k_{m-i}}]$.
Similarly, $[V_{i}, C_{k_{m-i}}, C_{k_{0}}q_{m-i}]$ would imply either $[C_{k_{0}}q_{m}, C_{k_{m-(i-1)}}, C_{k_{m-i}}]$ or 
\\ \noindent
$[C_{k_{0}}q_{m}, C_{k_{m-i}}, C_{k_{m-(i-1)}}]$.
Since we assumed that $[C_{k_{0}}q_{m}, C_{k_{j}}, C_{k_{l}}]$ never holds for $k_{j}\ne k_{l}$, then we must have
$[C_{k_{0}}q_{m}, V_{i}, C_{k_{m-i}}, C_{k_{0}}q_{m-(i+1)}]$.
We illustrate this in Figure \ref{fig proof P(a,b)}.
In particular, we note that $C_{k_{m-i}}$ lies on the geodesic in $\tree{\funalpha}$ from $C_{k_{0}}q_{m}$ to $C_{k_{0}}q_{m-(i+1)}$, that is, we have 
$[C_{k_{0}}q_{m}, C_{k_{m-i}}, C_{k_{0}}q_{m-(i+1)}]$.
Hence if the claim holds for $i\in\{1,\dots,m-1\}$, then it holds for $i+1$.
\begin{figure}[h]
\centering
\begin{tikzpicture}
\draw[red,thick] (-3,2) -- (1,2);
\draw[blue, thick] (-3,1) -- (4,1);
\draw[red,thick] (-3,-0.985) -- (0,-0.985);
\draw[blue,thick] (-3,-1.015) -- (0,-1.015);
\draw[red, thick] (-3,-0.98) -- (0,-0.98) -- (0.866,-0.48);
\draw[blue,thick] (-3,-1.02) -- (0,-1.02) -- (3.464,-3.02);
\draw[thick,Green] (0.886,-0.5) -- (0.02,-1) -- (3.47,-2.99175);
\filldraw (-3,2) circle [radius=0.0625cm];
\filldraw (-0.5,2) circle [radius=0.0625cm];
\filldraw[Green] (0,2) circle [radius=0.075cm];
\filldraw (1,2) circle [radius=0.0625cm];
\filldraw (-3,1) circle [radius=0.0625cm];
\filldraw[Green] (0,1) circle [radius=0.075cm];
\filldraw (0.5,1) circle [radius=0.0625cm];
\filldraw (4,1) circle [radius=0.0625cm];
\filldraw (-3,-1) circle [radius=0.0625cm];
\filldraw (-0.5,-1) circle [radius=0.0625cm];
\filldraw[Green] (0,-1) circle [radius=0.075cm];
\filldraw (0.866,-0.5) circle [radius=0.0625cm];
\filldraw (0.433,-1.25) circle [radius=0.0625cm];
\filldraw (3.434,-3) circle [radius=0.0625cm];
\draw[thick, -Triangle] (0,0.5) -- (0,-0.25);
\node at (-3,2.25) {\footnotesize{$C_{k_{0}}q_{m-i}$}};
\node at (-0.5,2.25) {\footnotesize{$C_{k_{m-(i-1)}}$}};
\node[Green] at (0,1.75) {\footnotesize{$V_{i}$}};
\node at (1.25,2.25) {\footnotesize{$C_{k_{0}}q_{m}$}};
\node at (-3,1.25) {\footnotesize{$C_{k_{0}}q_{m-i}$}};
\node[Green] at (0,0.75) {\footnotesize{$V_{i}$}};
\node at (0.75,1.25) {\footnotesize{$C_{k_{m-i}}$}};
\node at (4.5,1.25) {\footnotesize{$C_{k_{0}}q_{m-(i+1)}$}};
\node at (-3,-0.75) {\footnotesize{$C_{k_{0}}q_{m-i}$}};
\node at (-0.75,-0.75) {\footnotesize{$C_{k_{m-(i-1)}}$}};
\node[Green] at (0,-1.25) {\footnotesize{$V_{i}$}};
\node at (1.375,-0.3) {\footnotesize{$C_{k_{0}}q_{m}$}};
\node at (0.9,-1.1) {\footnotesize{$C_{k_{m-i}}$}};
\node at (3.75,-3.25) {\footnotesize{$C_{k_{0}}q_{m-(i+1)}$}};
\end{tikzpicture}
\caption{ Inductive Step in Proof of Lemma \ref{lemma P(a,b)} }
\label{fig proof P(a,b)}
\end{figure}
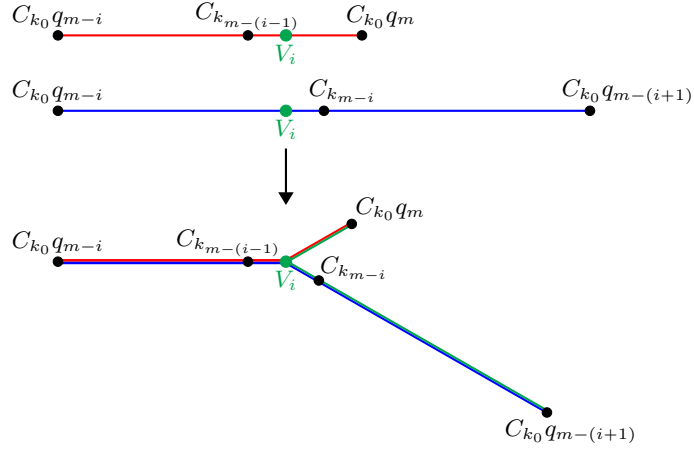

Since the claim holds for $i=1$, we conclude that the claim holds for all $i\in\{1,\dots,m\}$.
In particular, we have $[C_{k_{0}}q_{m}, C_{k_{1}}, C_{k_{0}}h_{0}]$.
However, $C_{k_{0}}h_{0} = C_{k_{0}}$, contradicting that we never have $[C_{k_{0}}q_{m}, C_{k_{j}}, C_{k_{l}}]$.
\end{proof}

\begin{lemma}\label{lemma translation by two implies elliptic}
Let $G_{i}\coset g_{i}$ be a $C$-vertex in $\tree{\funalpha}$ (that is, $i\in\{1,\dots,n\}$ and $g_{i}\in G$).
Let $U\coset x$ and $U \coset y$ be two neighbours of $G_{i}\coset g_{i}$ in $\tree{\funalpha}$ (that is, $x,y\in G$).
Then $x^{-1}y$ stabilises $G_{i}\coset g_{i}$ (i.e. $G_{i}\coset g_{i}x^{-1}y=G_{i}\coset g_{i}$ and $x^{-1}y\in G_{i}^{g_{i}}$).
\end{lemma}

\begin{proof}
Let $G_{i}\coset g_{i}$, $U\coset x$, and $U\coset y$ be as in the statement of the Lemma.
Since $G_{i}\coset g_{i}$ is adjacent in $\tree{\funalpha}$ to $U\coset x$, then $(G_{i}\coset g_{i})\cdot x^{-1}y$ must be adjacent in $\tree{\funalpha}$ to $(U\coset x)\cdot x^{-1}y=U\coset xx^{-1}y=U\coset y$.
Note that the neighbours of $U\coset y$ in $\tree{\funalpha}$ are the vertices $G_{1}\coset h_{1}, \dots, G_{n}\coset h_{n}$ for some $h_{1},\dots,h_{n}\in G$.
Then we must have that $(G_{i}\coset g_{i})\cdot x^{-1}y = G_{i}\coset h_{i}$.
But $G_{i}\coset g_{i}$ is also adjacent to $U\coset y$ in $\tree{\funalpha}$, and so we must also have that $G_{i}\coset g_{i}=G_{i}\coset h_{i}$.
That is, $(G_{i}\coset g_{i})\cdot x^{-1}y = G_{i}\coset g_{i}$ and so $x^{-1}y\in\stab_{G}(G_{i}\coset g_{i}) = G_{i}^{g_{i}}$.
\end{proof}

\subsection{`Volumes' of $\alpha$-Graphs}

In order to show that we can find paths in $\mathcal{C}_{n}$ from arbitrary $\alpha$-graph classes to $\class{\funalpha}$, we define quantities `volume' associated to $\alpha$-graphs, and show that this quantity is somehow `reducible' to that of $\funalpha=(\alpha:G_{1},\dots,G_{n})$.

\begin{defn} \label{defn spoke graph}
Let $\alpha_{0}=(\alpha:G_{1}^{g_{1}},\dots,G_{n}^{g_{n}})$ be a representative of an $\alpha$-graph class $\class{\alpha_{0}}$ in $\mathcal{S}_{n}^{(0)}$ and let $x\in G$.
We call the geodesic $[U\coset x,G_{i}\coset g_{i}]$ in $\tree{\funalpha}$ from the vertex $U\coset x$ to the vertex $G_{i}\coset g_{i}$ the \emph{ $i$\textsuperscript{th} $U\coset x$-spoke of $\alpha_{0}$ }.
We now construct the \emph{ $U\coset x$-spoke graph of $\alpha_{0}$ } by taking the disjoint union of all $n$ $U\coset x$-spokes of $\alpha_{0}$ and identifying them at the vertex $U\coset x$.
We denote this graph $\spoke{U\coset x}{\alpha_{0}}$. 
\end{defn}

\begin{defn} \label{defn vol of alpha graph}
Let $\alpha_{0}=(\alpha:G_{1}^{g_{1}},\dots,G_{n}^{g_{n}})$ and let $x\in G$.
We define the \emph{ $U\coset x$-volume } of $\alpha_{0}$, denoted $\vol{U\coset x}{\alpha_{0}}$, to be the total number of edges in the $U\coset x$-spoke graph $\spoke{U\coset x}{\alpha_{0}}$.
\end{defn}

Note that this definition is depenedent on the choice of $U$-vertex $U\coset x$ in $\tree{\funalpha}$ (equivalently, the choice of $x\in G$).

\begin{obs}
Let $\alpha_{0}=(\alpha:G_{1}^{g_{1}},\dots,G_{n}^{g_{n}})$ and $\alpha_{1}=(\alpha:G_{1}^{g_{1}h},\dots,G_{n}^{g_{n}h})$ be two equivalent $\mathfrak{S}$-labellings of an $\alpha$-graph class in $\mathcal{S}_{n}$, and let $x\in G$.
Since $G$ acts on $\tree{\funalpha}$ by isometries, we have that the $i$\textsuperscript{th} $U\coset xh$-spoke of $\alpha_{1}$ in $\tree{\funalpha}$ is the $h$-image of the $i$\textsuperscript{th} $U\coset x$-spoke of $\alpha_{0}$
(i.e. $[U\coset xh, G_{i}\coset g_{i}h]=[U\coset x,G_{i}\coset g_{i}]\cdot h$).
That is, $\spoke{U\coset xh}{\alpha_{1}}$ `looks like' $\spoke{U\coset x}{\alpha_{0}}$.
In particular, $\vol{U\coset xh}{\alpha_{1}}=\vol{U\coset x}{\alpha_{0}}$.
\end{obs}

Thus $U\coset x$-volume is a property of $\alpha$-graphs, rather than $\alpha$-graph classes, although volumes of equivalent $\alpha$-graphs are comparable by changing the choice of $x\in G$. One may adapt this to define a definitive volume of an $\alpha$-graph class (for example by taking a minimum or a sum over all $U\coset x$-volumes), but it will not be necessary for our arguments.

\begin{lemma}\label{lemma only funalpha has vol n}
Let $\alpha_{0}$ be a representative of an $\alpha$-graph class in $\mathcal{S}_{n}^{(0)}$.
We have that $\alpha_{0}\simeq\funalpha$ if and only if there is some $x\in G$ so that $\vol{U\coset x}{\alpha_{0}}=n$.
\end{lemma}

\begin{proof}
Recall that $\funalpha:=(\alpha:G_{1},\dots,G_{n})$.
Thus if $\alpha_{0}\simeq \funalpha$ then by Lemma \ref{lemma equivalent alpha labellings}, there exitsts $g\in G$ such that $\alpha_{0}=(G_{1}^{g},\dots,G_{n}^{g})$.
For each $i$ we have that $U\dash G_{i}$ is an edge in $\tree{\funalpha}$, thus so too is $U\coset g \dash G_{i}\coset g$.
We then have that $\spoke{U\coset g}{\alpha_{0}}$ is the star with central vertex $U\coset g$ and $n$ leaves $G_{1}\coset g, \dots, G_{n}\coset g$.
Hence $\vol{U\coset g}{\alpha_{0}}=n$.

Now let $\alpha_{0}=(\alpha:G_{1}^{g_{1}},\dots,G_{n}^{g_{n}})$ and suppose $\vol{U\coset x}{\alpha_{0}}=n$ for some $x\in G$.
Since $\spoke{U\coset x}{\alpha_{0}}$ comprises $n$ distinct $U\coset x$-spokes, we must have that each $U\coset x$-spoke has length $1$ in $\tree{\funalpha}$.
That is, for each $i$, $U\coset x \dash G_{i}\coset g_{i}$ is an edge in $\tree{\funalpha}$.
Then for each $i$, $U \dash G_{i}\coset g_{i}x^{-1}$ must also be an edge of $\tree{\funalpha}$.
However, the only neighbours of $U$ in $\tree{\funalpha}$ are the vertices $G_{1}, \dots, G_{n}$. 
So we must have that $G_{i}\coset g_{i}x^{-1}=G_{i}$ as cosets, and thus by Observation \ref{obs malnormalcy and coset equivalence}, $G_{i}^{g_{i}x^{-1}}=G_{i}$ .
Now $\alpha_{0}=(\alpha:G_{1}^{g_{1}},\dots,G_{n}^{g_{n}})\simeq(\alpha:G_{1}^{g_{1}x^{-1}},\dots,G_{n}^{g_{n}x^{-1}})=(\alpha:G_{1},\dots,G_{n})=\funalpha$.
\end{proof}

\begin{obs}\label{obs vol is integer at least n}
By choosing $U\coset x$ to be suitably `far away' from the vertices $G_{1}\coset g_{1}, \dots, G_{n}\coset g_{n}$ in $\tree{\funalpha}$, one can take $\vol{U\coset x}{\alpha_{0}}$ to be as large as desired.
However, since $\spoke{U\coset x}{\alpha_{0}}$ always comprises $n$ $U\coset x$-spokes, each of which have integer length at least 1, we deduce that $\vol{U\coset x}{\alpha_{0}}\ge n$ for any $\alpha_{0}$ and any $x\in G$ (and that $\vol{U\coset x}{\alpha_{0}}$ must always take integer values). 
\end{obs}


\begin{lemma}\label{lemma spokes fold}
Suppose $G_{1}^{g_{1}}\ast\dots\ast G_{n}^{g_{n}}$ is an $\mathfrak{S}$-free splitting of $G$.
Let $\alpha_{0}=(\alpha:G_{1}^{g_{1}},\dots,G_{n}^{g_{n}})$ and let $x\in G$ be such that $\vol{U\coset x}{\alpha_{0}}>n$.
Then there exist distinct $j$ and $l$ in $\{1,\dots,n\}$ so that the $l$\textsuperscript{th} $U\coset x$-spoke of $\alpha_{0}$ contains the $C$-vertex $G_{j}\coset g_{j}$.
That is, in $\tree{\funalpha}$ the vertex $G_{j}\coset g_{j}$ lies on the geodesic $[U\coset x,G_{l}\coset g_{l}]$, i.e. we have $[U\coset x, G_{j}\coset g_{j}, G_{l}\coset g_{l}]$.
\end{lemma}

\begin{proof}
Since $\vol{U\coset x}{\alpha_{0}}>n$, then the $U\coset x$-spoke graph of $\alpha_{0}$ contains more than $n$ edges.
In particular, there is some $i$ so that the $i$\textsuperscript{th} $U\coset x$-spoke of $\alpha_{0}$ in $\tree{\funalpha}$ contains more than one edge.
Then there exist $k_{0}\in\{1,\dots,n\}$ and $h\in G$ so that $G_{k_{0}}\coset g_{k_{0}}h$ lies on the $i$\textsuperscript{th} $U\coset x$-spoke, with $G_{k_{0}}\coset g_{k_{0}}h \ne G_{i}\coset g_{i}$.
Recall that $g_{k_{0}}$ is the exponent in $G_{k_{0}}^{g_{k_{0}}}$ in the $\mathfrak{S}$-labelling $\alpha_{0}$.
By Observation \ref{obs tree of funalpha properties}, we may assume without loss of generality that $G_{k_{0}}\coset g_{k_{0}}h$ is adjacent to $U\coset x$ in $\tree{\funalpha}$.

If $h\in G_{k_{0}}^{g_{k_{0}}}$ then $G_{k_{0}}\coset g_{k_{0}}h=G_{k_{0}}\coset g_{k_{0}}$, hence $G_{k_{0}}\coset g_{k_{0}}$ lies on the $i$\textsuperscript{th} $U\coset x$-spoke and the Lemma is satisfied
(in this case, the assumption that the $i$\textsuperscript{th} $U\coset x$-spoke has length strictly greater than 1 means we have that $G_{k_{0}}\coset g_{k_{0}} \ne G_{i}\coset g_{i}$).
We will thus assume $h\not\in G_{k_{0}}^{g_{k_{0}}}$.

Since $G_{1}^{g_{1}}\ast\dots\ast G_{n}^{g_{n}}$ is a free splitting for $G$, we may write $h=h_{1}\dots h_{m}$ where for each $a\in\{1,\dots,m\}$, $h_{a}$ belongs to a factor group $G_{k_{a}}^{g_{k_{a}}}$.
We will assume $g_{k_{0}}h_{1}\dots h_{m}$ is a reduced word --- that is, $k_{a}\ne k_{a-1}$ for each $a\in\{1,\dots,m\}$.

By Lemma \ref{lemma P(a,b)}, there exist $j$ and $l$ in $\{0,\dots,m\}$ with $k_{j}\ne k_{l}$ such that 
\\ \noindent $[G_{k_{0}}\coset g_{k_{0}}h, G_{k_{j}}\coset g_{k_{j}}, G_{k_{l}}\coset g_{k_{l}}]$ holds.
Note that $k_{j}\ne k_{l}$ means that $G_{k_{j}}\coset g_{k_{j}}\ne G_{k_{l}}\coset g_{k_{l}}$, and $h\not\in G_{k_{0}}^{g_{k_{0}}}$ means $G_{k_{0}}\coset g_{k_{0}}h\ne G_{k_{0}}\coset g_{k_{0}}$ and hence $G_{k_{0}}\coset g_{k_{0}}h\ne G_{k_{a}}\coset g_{k_{a}}$ for any $k_{a}\in\{1,\dots,n\}$.
In particular, $G_{k_{0}}\coset g_{k_{0}}h\ne G_{k_{j}}\coset g_{k_{j}}$.

Observe that since $G_{k_{0}}\coset g_{k_{0}}h$ is adjacent to $U\coset x$ in $\tree{\funalpha}$ and $[G_{k_{0}}\coset g_{k_{0}}h, G_{k_{j}}\coset g_{k_{j}}]$ has length at least 2 (since $\tree{\funalpha}$ is bipartite), then $[G_{k_{0}}\coset g_{k_{0}}h, G_{k_{j}}\coset g_{k_{j}}, G_{k_{l}}\coset g_{k_{l}}]$ implies either 
$[U\coset x, G_{k_{0}}\coset g_{k_{0}}h, G_{k_{j}}\coset g_{k_{j}}, G_{k_{l}}\coset g_{k_{l}}]$ or $[G_{k_{0}}\coset g_{k_{0}}h, U\coset x, G_{k_{j}}\coset g_{k_{j}}, G_{k_{l}}\coset g_{k_{l}}]$.
In either case, we have that the vertex $G_{k_{j}}\coset g_{k_{j}}$ in $\tree{\funalpha}$ lies on the $\tree{\funalpha}$-geodesic $[U\coset x,G_{k_{l}}\coset g_{k_{l}}]$, as required.
\end{proof}

We are now able to provide the key ingredient required for this Section: that $U\coset x$-volumes are `reducible':

\begin{lemma}\label{lemma volume is decreasable} 
Let $\class{\alpha_{0}}$ be an $\alpha$-graph class in $\mathcal{S}_{n}^{(0)}$ with representative $\alpha_{0}=(\alpha:G_{1}^{g_{1}},\dots,G_{n}^{g_{n}})$, and fix $x\in G$ so that $\vol{U\coset x}{\alpha_{0}}>n$.
Then there exist an $\alpha$-graph class $\class{\alpha_{1}}$ and an $A$-graph class $\class{A_{1}}$ in $\mathcal{S}_{n}^{(0)}$ so that $\class{\alpha_{0}} \dash \class{A_{1}} \dash \class{\alpha_{1}}$ is a path in $\mathcal{S}_{n}$, and $\class{\alpha_{1}}$ has a representative $\alpha_{1}$ with $\vol{U\coset x}{\alpha_{1}}<\vol{U\coset x}{\alpha_{0}}$.
\end{lemma}

\begin{proof}
Since $\vol{U\coset x}{\alpha_{0}}>n$ then by Lemma \ref{lemma spokes fold}, we have $[U\coset x, G_{i}\coset g_{i}, G_{j}\coset g_{j}]$ in $\tree{\funalpha}$ for some distinct $i,j\in\{1,\dots,n\}$.
Take $y$ and $z$ in $G$ such that $U\coset y$ and $U \coset z$ are adjacent to $G_{i}\coset g_{i}$ in $\tree{\funalpha}$ and satisfy $[U\coset x, U\coset y, G_{i}\coset g_{i}, U\coset z, G_{j}\coset g_{j}]$ (possibly with $U\coset y=U\coset x$, but never $U\coset y=U\coset z$).
Set $\alpha_{1}:=(\alpha:G_{1}^{g_{1}}, \dots, G_{j-1}^{g_{j-1}}, G_{j}^{g_{j}z^{-1}y}, G_{j+1}^{g_{j+1}}, \dots, G_{n}^{g_{n}})$ and $A_{1}:=(A:G_{i}^{g_{i}}; G_{1}^{g_{1}}, \dots, G_{j}^{g_{j}}, \dots, G_{n}^{g_{n}})$.
By Lemma \ref{lemma translation by two implies elliptic}, we have that $z^{-1}y\in G_{i}^{g_{i}}$, and so by Lemma \ref{lemma equivalent A labellings} we have $(A:G_{i}^{g_{i}}; G_{1}^{g_{1}}, \dots, G_{j}^{g_{j}}, \dots, G_{n}^{g_{n}})\simeq(A:G_{i}^{g_{i}}; G_{1}^{g_{1}}, \dots, G_{j}^{g_{j}z^{-1}y}, \dots, G_{n}^{g_{n}})$.
We now see that by Lemma \ref{lemma alphas collapse to As}, $\class{\alpha_{0}} \dash \class{A_{1}} \dash \class{\alpha_{1}}$ is a path in $\mathcal{S}_{n}$.

Observe that $\spoke{U\coset x}{\alpha_{0}}$ and $\spoke{U\coset x}{\alpha_{1}}$ differ only in the $j$\textsuperscript{th} $U\coset x$-spoke.
Let $V$ be the vertex in $\tree{\funalpha}$ satisfying $[U\coset y, U\coset x] \cap [U\coset y, G_{j}\coset g_{j}z^{-1}y] = [U\coset y, V]$.
Note that we may have $V\in\{ U\coset y, U\coset x, G_{j}\coset g_{j}z^{-1}y\}$ (or we may have $V\not\in\{ U\coset y, U\coset x, G_{j}\coset g_{j}z^{-1}y\}$).
For brevity, let $k:=|[U\coset z, G_{j}\coset g_{j}]|$, $l:=|[U\coset y, V]|$, and $m:=|[U\coset x, U\coset y]|$.
We illustrate this in Figure \ref{fig proof volume is decreasable}.

\begin{figure}[h]
\centering
\begin{tikzpicture}
\draw[blue, thick]	(-1.5,0.025)	--	(0,0.025)					;
\draw[black,thick]	(0,0)		--	(2,0)						;
\draw[Green,thick]	(2,0)		--	(5,0)						;
\draw[red,thick]	(-4,0)		--	(0,0)						;
\draw[Green,thick] (0,-0.025)	--	(-1.5,-0.025)	--	(-2.8,-0.75)	;
\filldraw	(-4,0)		circle	[radius=0.06cm]	;
\filldraw	(-1.5,0)	circle	[radius=0.06cm]	;
\filldraw	(0,0)		circle	[radius=0.06cm]	;
\filldraw	(1,0)		circle	[radius=0.06cm]	;
\filldraw	(2,0)		circle	[radius=0.06cm]	;
\filldraw	(5,0)		circle	[radius=0.06cm]	;
\filldraw	(-2.8,-0.75)	circle	[radius=0.06cm]	;
\node at (-4,0.25)		{\small{$U\coset x$}}			;
\node at (-1.5,0.25)	{\small{$V$}}				;
\node at (0,0.25)		{\small{$U\coset y$}}			;
\node at (1,0.25)		{\small{$G_{i}\coset g_{i}$}}		;
\node at (2,0.25)		{\small{$U\coset z$}}			;
\node at (5,0.25)		{\small{$G_{j}\coset g_{j}$}}		;
\node at (-3.75,-0.75)	{\small{$G_{j}\coset g_{j}z^{-1}y$}}	;
\node[red] at 	(-2,-0.15)		{\footnotesize{$m$}}	;
\node[blue] at 	(-0.75,-0.2)		{\footnotesize{$l$}}	;
\node[black] at 	(0.5,-0.2)		{\scriptsize{$1$}}		;
\node[black] at 	(1.5,-0.2)		{\scriptsize{$1$}}		;
\node[Green] at 	(3.5,-0.2)		{\footnotesize{$k$}}	;
\node[Green] at 	(-1.45,-0.225)	{\footnotesize{$k$}}	;
\end{tikzpicture}
\caption{$j$\textsuperscript{th} $U\coset x$-Spokes of $\alpha_{0}$ and $\alpha_{1}$ in $\tree{\funalpha}$}
\label{fig proof volume is decreasable}
\end{figure}
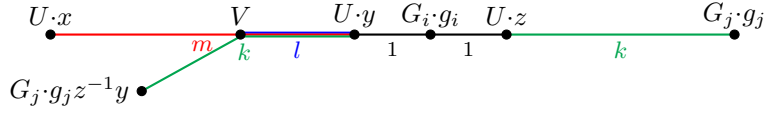

Note that $0\le l \le \min\{k,m\}$.
Since $G$ acts on $\tree{\funalpha}$ by isometries, then $|[U\coset y, G_{j}\coset g_{j}z^{-1}y]|=|[U\coset z, G_{j}\coset g_{j}]|=k$.
We see that $|[U\coset x, G_{j}\coset g_{j}]|= m+k+2$, while $|[U\coset x, G_{j}\coset g_{j}z^{-1}y]|=m+k-2l$.
Thus $\vol{U\coset x}{\alpha_{1}} = \vol{U\coset x}{\alpha_{0}}-2-2l \le \vol{U\coset x}{\alpha_{0}}-2 < \vol{U\coset x}{\alpha_{0}}$, as required.
\end{proof}

Finally, we can prove our main result of this Section: that paths to $\class{\funalpha}$ always exist in $\mathcal{S}_{n}$.

\begin{prop}\label{prop alpha-A paths}
If  $\class{\alpha_{0}}\in\mathcal{S}_{n}^{(0)}$ is an $\alpha$-graph class then there exists a path in $\mathcal{S}_{n}$ of the form: $$\class{\alpha_{0}} \dash \class{A_{1}} \dash \class{\alpha_{1}} \dash \dots \dash \class{\alpha_{m-1}} \dash \class{A_{m}} \dash \class{\alpha_{m}}=\class{\funalpha}$$
\end{prop}

\begin{proof}
Let $\class{\alpha_{0}}$ be an arbitrary $\alpha$-graph class in $\mathcal{S}_{n}$ with representative $\alpha_{0}$.
We will prove by (strong) induction on $\vol{U}{\alpha_{0}}$ that there is some (finite) path in $\mathcal{S}_{n}$ from $\class{\alpha_{0}}$ to $\class{\funalpha}$.
Recall from Observation \ref{obs vol is integer at least n} that $\vol{U}{\alpha_{0}}$ is an integer with $\vol{U}{\alpha_{0}}\ge n$.

Suppose $\vol{U}{\alpha_{0}}=n$.
By Lemma \ref{lemma only funalpha has vol n}, we then have $\alpha_{0}\simeq\funalpha$.
Thus $\class{\alpha_{0}}=\class{\funalpha}$ is a path (of length 0) and we are done.

Fix $N\ge n$ and suppose that for all $\alpha$-graphs $\hat{\alpha}$ with $\vol{U}{\hat{\alpha}}\le N$ there is some path $p$ in $\mathcal{S}_{n}$ from $\class{\hat{\alpha}}$ to $\class{\funalpha}$. 

Now suppose $\vol{U}{\alpha_{0}}=N+1$. Then $\vol{U}{\alpha_{0}}>N\ge n$, and by Lemma \ref{lemma volume is decreasable}, there exist $\class{A_{1}},\class{\alpha_{1}}\in\mathcal{S}_{n}^{(0)}$ so that $\class{\alpha_{0}} \dash \class{A_{1}} \dash \class{\alpha_{1}}$ is a path in $\mathcal{S}_{n}$ and $\vol{U}{\alpha_{1}}<\vol{U}{\alpha_{0}}=N+1$.
In particular, $\vol{U}{\alpha_{1}}\le N$, and so by hypothesis there is some path $p$ in $\mathcal{S}_{n}$ from $\class{\alpha_{1}}$ to $\class{\funalpha}$. 
Thus by concatenating the paths $\class{\alpha_{0}} \dash \class{A_{1}} \dash \class{\alpha_{1}}$ and $p$, we have found a path in $\mathcal{S}_{n}$ from $\class{\alpha_{0}}$ to $\class{\funalpha}$. 

Since the statement holds for $\vol{U}{\alpha_{0}}=n$, and the statement holding for all $\vol{U}{\alpha_{0}}\le N$ (for some $N\ge n$) implies it holds for $\vol{U}{\alpha_{0}}=N+1$, we conclude that the statement holds for all possible values of $\vol{U}{\alpha_{0}}$.
Thus given an arbitrary $\alpha$-graph class $\class{\alpha_{0}}$, we can find a path in $\mathcal{S}_{n}$ from $\class{\alpha_{0}}$ to $\class{\funalpha}$.

That our path is of the form $\class{\alpha_{0}} \dash \class{A_{1}} \dash \class{\alpha_{1}} \dash \dots \dash \class{\alpha_{m-1}} \dash \class{A_{m}} \dash \class{\alpha_{m}}=\class{\funalpha}$ follows immediately by recalling from Observation \ref{obs Sn properties} that $\mathcal{S}_{n}$ is a bipartite graph between the set of $\alpha$-graph classes and the set of $A$-graph classes.
\end{proof}

\connected

\begin{proof}
Recall from Observation \ref{obs Sn properties} that every edge of $\mathcal{S}_{n}$ is of the form $\class{\alpha_{0}}\dash\class{A_{0}}$.
Moreover, from Observation \ref{obs fun dom for Sn} we have that every $A$-graph class $\class{A_{1}}$ is adjacent in $\mathcal{S}_{n}$ to some $\alpha$-graph class $\class{\alpha_{1}}$.
Thus by Proposition \ref{prop alpha-A paths}, given any point $P$ in $\mathcal{S}_{n}$, there is a path in $\mathcal{S}_{n}$ from $P$ to the vertex $\class{\funalpha}$.
Hence $\mathcal{S}_{n}$ is path connected.
\end{proof}


\section{Generators for $\outs(G_{1}\ast\dots\ast G_{n})$}\label{section generators}

Let $G=G_{1}\ast\dots\ast G_{n}$ and $\mathfrak{S}=(G_{1},\dots,G_{n})$.
Recall that $\funalpha$ is the $\mathfrak{S}$-labelling $(\alpha:G_{1},\dots,G_{n})$ of the graph $\alpha$ in Figure \ref{fig alpha graph shape}
and $\funA{i}$ is the $\mathfrak{S}$-labelling $(A:G_{i};G_{1},\dots,G_{n})=(A:G_{i},G_{1},\dots,G_{i-1},G_{i+1},\dots,G_{n})$ of the graph $A$ in Figure \ref{fig A graph shape}.

The purpose of this Section (and the paper as a whole) is to prove Theorem \ref{main theorem}, that is, to show that any pure symmetric outer automorphism (Definition \ref{defn pure symmetric autos}) of the splitting $G_{1}\ast\dots\ast G_{n}$ is a product of factor outer automorphisms (Definiton \ref{defn factor autos}) and Whitehead outer automorphisms (Definition \ref{defn whitehead autos}) relative to $\mathfrak{S}$.
Equivalently, we will show that $\outs(G)$ is generated by factor outer automorphisms and Whitehead outer automorphisms relative to $\mathfrak{S}$. 

We begin by determining the $\outs(G)$ stabilisers of the vertices $\class{\funalpha}$ and $\class{\funA{i}}$ in $\mathcal{S}_{n}$.

\begin{prop}\label{prop stab funalpha}
Let $\Phi\in\outs(G)$.
Then $\class{\funalpha}\cdot\Phi=\class{\funalpha}$ if and only if $\Phi$ is a factor outer automorphism.
\end{prop}

\begin{proof}
If $\Phi\in\outs(G)$ is a factor outer automorphism,
then there exists $\varphi\in\Phi$ which is a factor automorphism
and by Observation \ref{obs factor and inner autos with labellings}, we have:
\begin{align*}
\class{\funalpha}\cdot\Phi	&	=	\class{(\alpha:G_{1},\dots,G_{n})}\cdot\Phi	\\
				&	=	\class{(\alpha:\varphi(G_{1}),\dots,\varphi(G_{n}))}	\\
				&	=	\class{(\alpha:G_{1},\dots,G_{n})}	\\
				&	=	\class{\funalpha}
\end{align*}

Now suppose $\Phi\in\outs(G)$ is such that $\class{\funalpha}\cdot\Phi=\class{\funalpha}$. Then for any $\varphi\in\Phi$, we have $(\alpha:\varphi(G_{1}),\dots,\varphi(G_{n}))\simeq(\alpha:G_{1},\dots,G_{n})$.
Fix $\varphi\in\Phi$.
By Lemma \ref{lemma equivalent alpha labellings} we deduce that there exists some $g\in G$ so that for each $i\in\{1,\dots,n\}$, $\varphi(G_{i})=G_{i}^{g}$.
Let $\iota_{g^{-1}}\in\aut(G)$ be the inner automorphism $x\mapsto gxg^{-1}$ for all $x\in G$.
Then $\iota_{g^{-1}}\varphi\in\Phi$ and for each $i\in\{1,\dots,n\}$, we have $\iota_{g^{-1}}(\varphi(G_{i}))=\iota_{g^{-1}}(G_{i}^{g})=G_{i}$.
Letting $\varphi_{i}$ be the map $\iota_{g^{-1}}\varphi$ with domain restricted to the subgroup $G_{i}$, we have that $\varphi_{i}\in\aut(G_{i})$.
Hence $(\varphi_{1},\dots,\varphi_{n})\in\factorautos$ so $\iota_{g}\varphi$ is a factor automorphism, and thus $\Phi$ is a factor outer automorphism.
\end{proof}

\begin{prop}\label{prop stab funAi}
Let $\Psi\in\outs(G)$.
Then $\class{\funA{i}}\cdot\Psi=\class{\funA{i}}$ if and only if $\Psi$ can be written as a product of factor outer automorphisms and Whitehead outer automorphisms with operating factor $G_{i}$.
\end{prop}

\begin{proof}
If $\Phi\in\outs(G)$ is a factor outer automorphism,
then there exists $\varphi\in\Phi$ which is a factor automorphism
and by Observation \ref{obs factor and inner autos with labellings}, we have:
\begin{align*}
\class{\funA{i}}\cdot\Phi	&	=	\class{(A:G_{i};G_{1},\dots,G_{n})}\cdot\Phi	\\
				&	=	\class{(A:\varphi(G_{i});\varphi(G_{1}),\dots,\varphi(G_{n}))}	\\
				&	=	\class{(A:G_{i};G_{1},\dots,G_{n})}	\\
				&	=	\class{\funA{i}}
\end{align*}
If $\Omega\in\outs(G)$ is a Whitehead outer automorphism with operating factor $G_{i}$,
then there is a (unique) $\omega\in\Omega$ which is a Whitehead automorphism with operating factor $G_{i}$.
Now for some $j\in\{1,\dots,n\}-\{i\}$ and some $x\in G_{i}$, we have $\omega=(G_{j},x)$.
Setting $g_{j}=x$ and $g_{k}=1$ for $k\in\{1,\dots,n\}-\{j\}$, we have that for each $k\in\{1,\dots,n\}$, $\omega(G_{k})=G_{k}^{g_{k}}$.
Note that $g_{1},\dots,g_{n}\in G_{i}$, so by Lemma \ref{lemma equivalent A labellings}, we have:
\begin{align*}
\class{\funA{i}}\cdot\Omega	&	=	\class{(A:\omega(G_{i});\omega(G_{1}),\dots,\omega(G_{n}))}	\\
					&	=	\class{(A:G_{i}^{g_{i}};G_{1}^{g_{1}},\dots,G_{n}^{g_{n}})}	\\
					&	=	\class{(A:G_{i};G_{1},\dots,G_{n})}	\\
					&	=	\class{\funA{i}}
\end{align*}
Therefore any factor outer automorphism or Whitehead outer automorphism with operating factor $G_{i}$ stabilises $\class{\funA{i}}$, and hence so too must any product of these.

\smallskip
Now suppose $\Psi\in\outs(G)$ is such that $\class{\funA{i}}\cdot\Psi=\class{\funA{i}}$.
Then for all $\psi\in\Psi$ we have $(A:\psi(G_{i});\psi(G_{1}),\dots,\psi(G_{n})) \simeq (A:G_{i};G_{1},\dots,G_{n}) = \funA{i}$.
By Lemma \ref{lemma equivalent A labellings} and Observation \ref{obs factor and inner autos with labellings} there is some $\psi\in\Psi$ such that there exist $g_{1},\dots,g_{n}\in G_{i}$ with $\psi(G_{j})=G_{j}^{g_{j}}$ for all $j\in\{1,\dots,n\}$.

Fix $j\in\{1,\dots,n\}$. Since $\psi(G_{j})=G_{j}^{g_{j}}$, then for each $x\in G_{j}$ there exists $y_{x}\in G_{j}^{g_{j}}$ such that $\psi(x)=y_{x}$
(and for each $z\in G_{j}^{g_{j}}$ there exists $x\in G_{j}$ with $z=\psi(x)$).
Define a map $\varphi_{j}$ with domain $G_{j}$ by $x\mapsto g_{j}y_{x}g_{j}^{-1}$ for all $x\in G_{j}$.
Note that since $y_{x}\in G_{j}^{g_{j}}$ then $g_{j}y_{x}g_{j}^{-1}\in G_{j}$.
Thus $\varphi_{j}\in\aut(G_{j})$.
Now $\psi(x)=g_{j}^{-1}\varphi_{j}(x)g_{j}=(G_{j},g_{j})(\varphi_{j}(x))$ where $(G_{j},g_{j})$ is a Whitehead automorphism with operating factor $G_{i}$ for all $j\ne i$ and $(G_{i},g_{i})\in\aut(G_{i})$ is the inner automorphism of $G_{i}$ which conjugates each element of $G_{i}$ by $g_{i}$.
Identifying $\varphi_{j}$ (and $(G_{i},g_{i})$ when $j=i$) with its image in $\factorautos$ under the natural inclusion $\aut(G_{j})\hookrightarrow\factorautos$,
we note that if $z\in G_{k}$ for any $k\ne j$ then $(G_{j},g_{j})(z)=z=\varphi_{j}(z)$.

Since this applies for each $j\in\{1,\dots,n\}$, we conclude that $\psi=(G_{1},g_{1})\varphi_{1}\dots(G_{n},g_{n})\varphi_{n}$ is a product of factor automorphisms and Whitehead automorphisms.
Thus $\Psi=[\psi]=[(G_{1},g_{1})\varphi_{1}\dots(G_{n},g_{n})\varphi_{n}]=[(G_{1},g_{1})][\varphi_{1}]\dots[(G_{n},g_{n})][\varphi_{n}]$ is a product of factor outer automorphisms and Whitehead outer automorphisms.
\end{proof}

In Proposition \ref{prop alpha-A paths}, we established the existence of paths in $\mathcal{S}_{n}$ from arbitrary $\alpha$-graph classes to $\class{\funalpha}$.
We will now show how these paths can be used to generate (outer) automorphisms.

\begin{prop}\label{prop inductive step for generators}
Let $\class{\funalpha}=\class{\alpha_{0}}\dash\class{A_{1}}\dash\class{\alpha_{1}}\dash\class{A_{2}}\dash\dots\dash\class{A_{m}}\dash\class{\alpha_{m}}$ be a path in $\mathcal{S}_{n}$. Then there exists $\Psi\in\outs(G)$ with $\class{\alpha_{m}}=\class{\funalpha}\cdot\Psi$ such that $\Psi$ is a product of factor outer automorphisms and Whitehead outer automorphisms.
\end{prop}

\begin{proof}
We prove the statement of the Lemma by induction on $m$.

First, suppose $m=0$. Then our path consists of a single point $\class{\funalpha}=\class{\alpha_{0}}$, and by Proposition \ref{prop stab funalpha}, any factor outer automorphism $\Psi$ will satisfy $\class{\funalpha}\cdot\Psi=\class{\funalpha}=\class{\alpha_{0}}$. In particular, the outer automorphism class of the identity automorphism is such a factor outer automorphism.

Now let $N>0$ and suppose the statement holds for all such paths with $m=N$.
We will show that the statement must also hold for all such paths with $m=N+1$.

Consider an arbitrary path $$\class{\funalpha}=\class{\alpha_{0}}\dash\class{A_{1}}\dash\class{\alpha_{1}}\dash\class{A_{2}}\dash\dots\dash\class{A_{N}}\dash\class{\alpha_{N}}\dash\class{A_{N+1}}\dash\class{\alpha_{N+1}}$$ in $\mathcal{S}_{n}$.
By hypothesis, there is some $\Phi\in\outs(G)$ with $\class{\alpha_{N}}=\class{\funalpha}\cdot\Phi$ such that $\Phi$ is a product of factor outer automorphisms and Whitehead outer automorphisms.
Since $\class{\alpha_{N}}\dash\class{A_{N+1}}\dash\class{\alpha_{N+1}}$ is a path in $\mathcal{S}_{n}$, then by Observation \ref{obs action preserves adjacency} so too is $\class{\funalpha}=\class{\alpha_{N}}\cdot\Phi^{-1}\dash\class{A_{N+1}}\cdot\Phi^{-1}\dash\class{\alpha_{N+1}}\cdot\Phi^{-1}$.
By Observation \ref{obs collapses of funalpha} we deduce that $\class{A_{N+1}}\cdot\Phi^{-1}=\class{\funA{i}}$ for some $i\in\{1,\dots,n\}$.
Now by Lemma \ref{lemma alphas differ by psi in stab A} there exists $\Omega\in\outs(G)$ with $\class{\alpha_{N+1}}\cdot\Phi^{-1}=\class{\funalpha}\cdot\Omega$ and $\class{A_{N+1}}\cdot\Omega=\class{A_{N+1}}$, moreover from Proposition \ref{prop stab funAi} we see that $\Omega$ is a product of factor outer automorphisms and Whitehead outer automorphisms.
Hence $\class{\alpha_{N+1}}=\class{\funalpha}\cdot\Omega\Phi$ where $\Omega\Phi\in\outs(G)$ is a product of factor outer automorphisms and Whitehead outer automorphisms, as required.
\end{proof}


\maintheorem

\begin{proof}
Let $\Psi\in\outs(G)$ and set $\class{\alpha_{0}}:=\class{\funalpha}\cdot\Psi\in\mathcal{S}_{n}$.
By Proposition \ref{prop alpha-A paths}, there exists some edge path 
$$\class{\alpha_{0}} \dash \class{A_{1}} \dash \class{\alpha_{1}} \dash \class{A_{2}} \dash \dots \dash \class{A_{m}} \dash \class{\alpha_{m}}=\class{\funalpha}$$
in $\mathcal{S}_{n}$ (where each $\class{\alpha_{i}}$ is an $\alpha$-graph class and each $\class{A_{i}}$ is an $A$-graph class).
Then by Proposition \ref{prop inductive step for generators} there exists $\Theta\in\outs(G)$ with $\class{\alpha_{0}}=\class{\funalpha}\cdot\Theta$ such that $\Theta$ is a product of factor outer automorphisms and Whitehead outer automorphisms.
Now $\class{\funalpha}\cdot\Psi\Theta^{-1}=\class{\funalpha}$ so by Proposition \ref{prop stab funalpha} there is some factor outer automorphism $\Phi\in\outs(G)$ with $\Psi\Theta^{-1}=\Phi$.
Hence $\Psi=\Phi\Theta$ is a product of factor outer automorphisms and Whitehead outer automorphisms.
\end{proof}

\begin{cor}
If $G_{1}\ast\dots\ast G_{n}$ is a Grushko decomposition for $G$ and moreover the factor groups $G_{i}$ are pairwise non-isomorphic, then $\out(G)$ is generated by factor outer automorphisms and Whitehead outer automorphisms.
\end{cor}

\begin{proof}
This follows immediately by considering Theorem \ref{main theorem} in the context of Observation \ref{obs Grushko decomposition}.
\end{proof}


\bibliography{references}
\bibliographystyle{abbrv}


\end{document}